\documentclass[12pt]{amsart}

\usepackage[margin=3cm]{geometry}
\usepackage{graphicx}
\usepackage{amssymb}
\usepackage{epstopdf}
\usepackage{enumerate}

\newtheorem{thm}{Theorem}[section]
\newtheorem{lemma}[thm]{Lemma}

\newtheorem{prop}[thm]{Proposition}
\newtheorem{cor}[thm]{Corollary}
\newtheorem{remark}[thm]{Remark}
\newtheorem*{ack}{Acknowledgement}

\begin{document}

\title[Escape Rates]{Escape Rates for Gibbs measures}

\author{Andrew Ferguson}
\address{Andrew Ferguson\\Mathematics Institute\\Zeeman Building\\
University of Warwick\\
Coventry\\CV4 7AL\\UK.}
\email{a.j.ferguson@warwick.ac.uk}

\author{Mark Pollicott}
\address{Mark Pollicott\\Mathematics Institute\\Zeeman Building\\
University of Warwick\\
Coventry\\CV4 7AL\\UK.}
\email{m.pollicott@warwick.ac.uk}


\date{\today}

\maketitle

\begin{abstract}In this paper we study the asymptotic behaviour of the escape rate of a Gibbs measure supported on a conformal repeller through a small hole.  There are additional applications to the convergence of Hausdorff dimension of the survivor set.\end{abstract}

\section{Introduction}

Given any  transformation $T: X \to X$ preserving an ergodic probability measure $\mu$ and any Borel set $A \subset X$ the escape rate quantifies the asymptotic behaviour of the measure of the set of points $x\in X$ for which none of the first $n$ terms in the orbit intersect $U$.  Bunimovich and  Yurchenko \cite{BunYur08} considered the fundamental case of the doubling map and Haar measure, and where $U$ is a dyadic interval.  Subsequently, Keller and Liverani \cite{KelLiv09} proved a  general perturbation result which, provided the correct functional setup holds, shows a similar formula holds.  It was then shown that these hypotheses hold when $T$ is an expanding interval map and $\mu$ the absolutely continuous invariant probability measure.  Other papers related to this this topic include \cite{BahBos10, BunDet07, DemYou06, LivMau03} and reference therein.

In this paper, we prove analogous results to those found in \cite{BunYur08} in the more general setting of Gibbs measures supported on conformal repellers.  Much of the analysis is undertaken in the setting of subshifts of finite type, this not only allows us to prove similar results for a broad class of maps which can be modelled symbolically but also improve on the work of Lind \cite{Lind89} who considered the convergence of topological entropy for a topologically mixing subshift.     

Another interesting aspect of our analysis is the connection  with the work of Hirata \cite{Hir93} on the  exponential law for first return times for Axiom A diffeomorphisms.  Some of the ingredients in our approach were suggested by Hirata's paper, although we had to significantly modify the actual details.

Let $\mathcal{M}$ be a Riemannian manifold and $f:\mathcal{M}\to\mathcal{M}$ a $C^{1}$-map.  Let $J$ be a compact subset of $\mathcal{M}$ such that $f(J)=J$.  We say that the pair $(J,f)$ is a conformal repeller if
\begin{enumerate}
\item $f|_{J}$ is a conformal map.
\item there exists $c>0$ and $\lambda>1$ such that $\|df^{n}_x v\|\geq c\lambda^n\|v\|$ for all $x\in J$, $v\in T_x\mathcal{M}$, and $n\geq 1$. 	
\item $f$ is topologically mixing on $J$.
\item $J$ is maximal, i.e. there exists an open neighbourhood $V\supset J$ such that
	\begin{equation}\nonumber J=\{x\in V\,:\,f^{n}(x)\in V\,\,\text{for all }n\geq 0\}.\end{equation}
\end{enumerate}

Let $\phi:J\to\mathbb{R}$ be $\alpha$-H\"older and let $\mu$ denote the associated equilibrium state, i.e.
\begin{equation}\nonumber P(\phi)=\sup\left\{h_{\nu}(f)+\int \phi d\nu\,:\,f_*(\nu)=\nu,\,\nu(J)=1\right\}=h_{\mu}(f)+\int \phi d\mu,\end{equation}
where $h_{\nu}(f)$ denotes the Kologomorov-Sinai entropy of the measure $\nu$ (see \cite{Wal82} for further details).  

Fix $z\in J$ for $\epsilon>0$ we define the escape rate of $\mu$ through $B(z,\epsilon)$, (i.e. the rate at which mass `escapes' or `leaks' through the hole $B(z,\epsilon)$), by

\begin{equation}\nonumber r_\mu(B(z,\epsilon))=-\limsup_{k\to\infty}\frac{1}{k}\log \mu\{x\in J\,:\,f^i(x)\not\in B(z,\epsilon), 0\leq i \leq k-1\}.\end{equation}

Our first result concerns the asymptotic behaviour of $r_\mu(B(z,\epsilon))$ for small $\epsilon$.

\begin{thm}Let $(J,f)$ be a conformal repeller, $\phi:J\to\mathbb{R}$ H\"older continuous, and $\mu$ the associated equilibrium state, fix $z\in J$, then

	\begin{equation}
	\nonumber
	\lim_{\epsilon\rightarrow 0} \frac{r_\mu (B(z,\epsilon)) }{\mu(B(z,\epsilon)) }= d_\phi(z)=
	\begin{cases} 1 &\text{if }z \text{ is not periodic} \\
	1-e^{\phi^p(z)-p P(\phi)}  & \text{if }z\text{ has prime period p}\end{cases}\end{equation}
	where $\phi^p(z)=\phi(z)+\phi(f(z))+\cdots+\phi(f^{p-1}(z)).$\label{escratesexp}\end{thm}
  
We also obtain an asymptotic formula the Hausdorff dimension of the survivor set: 

\begin{equation}\nonumber J_{\epsilon}=\{x\in J\,:\,f^{k}(x)\not\in B(z,\epsilon), \text{for all }k\geq 0\},\end{equation}
i.e. all points whose orbits are $\epsilon$-bounded away from $z$.

Suppose now that $f\in C^{1+\alpha}(J)$ for some $\alpha>0$.  Let $\mu$ denote the equilibrium state related to the potential $\phi=-s\log|f^\prime|$, where $s={\rm dim}_H(J)$.  For $\epsilon>0$ we let $s_{\epsilon}$ denote the Hausdorff dimension of the set $J_{\epsilon}$.

\begin{thm}Let $(J,f)$ be a conformal repeller with $f\in C^{1+\alpha}(J)$.  Let $\phi=-s\log|f^\prime|$ and let $\mu$ denote the associated equilibrium state.  Fix $z\in J$, then
\begin{equation}\nonumber \lim_{\epsilon\to 0}\frac{s-s_\epsilon}{\mu(B(z,\epsilon))}=\frac{d_\phi(z)}{\int \log|f^\prime|d\mu}.\end{equation}	\label{dimension}
	\end{thm}
	
\begin{remark}A similar formula was obtained by Hensley \cite{Hen92} in the setting of continued fractions.\end{remark}
	
The paper is structured as follows: in section 2 we apply the Theorems \ref{escratesexp} and \ref{dimension} to concrete examples.  In section 3 we study the spectral properties of transfer operators acting on a certain class of Banach spaces.  Section 4 contains a perturbation result, while in section 5 we prove the result in the analogue of Theorem \ref{escratesexp} in the setting of subshifts of finite type.  Finally sections 6 and 7 contain the proofs of Theorems \ref{escratesexp} and \ref{dimension} respectively.

\section{Examples}
To illustrate the main results we briefly consider two simple examples.

\subsection{Hyperbolic Julia sets}

Let $f:\widehat{\mathbb{C}}\to\widehat{\mathbb{C}}$ be a rational map of degree $d\geq 2$, where $\hat{\mathbb{C}}$ denotes the Riemann sphere.  The Julia set of $R$ is the closure of the repelling periodic points of $f$, i.e.

\begin{equation}\nonumber J=cl\left(\{z\in\widehat{\mathbb{C}}\,:\,f^p(z)=z,\text{ for some  }p\geq 1 \text{ and }|(f^p)^\prime (z)|>1\}\right).\end{equation}

The map $f:J\to J$ is a conformal expanding map and the results of the previous section apply.  As an example, the map $f(z)=z^2+c$ for $|c|<1/4$ is hyperbolic.  Define $\phi:J\to\mathbb{R}$ by $\phi(z)=-s\log|2z|$, where $s$ denotes the Hausdorff dimension of $J$.  Let $\mu$ denote the associated equilibrium state.  Setting $z=\frac{1+\sqrt{1-4c}}{2}$, then we see that $f(z)=z$ and $|f^{'}(z)|>1$ and accordingly Theorem \ref{escratesexp} implies that
\begin{equation}\nonumber \lim_{\epsilon\to 0}\frac{r_\mu(B(z,\epsilon))}{\mu(B(z,\epsilon))}=1-\frac{1}{|2z|^{s}}.\end{equation}

\subsection{One dimensional Markov Maps}

Assume that there exists a finite family of disjoint closed intervals $I_1,I_2,\ldots,I_m\subset [0,1]$ and a $C^{1+\alpha}$ map $f:\bigcup_{i} I_i\to [0,1]$ such that
\begin{enumerate}
	\item for every $i$, there is a subset $P=P(i)$ of indices with $f(I_i)\cap\bigcup_{i} I_i=\bigcup_{k\in P}I_k$.
	\item for every $x\in \cup_i int(I_i)$, the derivative of $f$ satisfies $|f^\prime(x)| \geq \rho$ for some fixed $\rho>0$.
	\item there exists $\lambda>1$ and $n_0>0$ such that if $f^m(x)\in\cup_i I_i$, for all $0\leq m \leq n_0-1$ then $|(f^{n_0})^\prime(x)|\geq \lambda$.\end{enumerate}
Let $J=\{x\in [0,1]: f^n(x)\in\cup_i I_i \text{ for all }n\in\mathbb{N}\}$.  The set $J$ is a repeller for the map $f$ and conformality follows from the domain being one-dimensional.

If we take $I_1=[0,1/3]$, $I_2=[2/3,1]$ and let $f(x)=3x ({\rm mod} 1)$, the associated repeller $J$ is the middle-third Cantor set.  
Let $z=1/4$, then $z\in J$ and has prime period $2$.
Set $\phi(x)=-\log(2)$, and let $\mu$ denote the associated equilibrium state. Then Theorem \ref{escratesexp} implies that

\begin{equation}\nonumber \lim_{\epsilon\to 0}\frac{r_\mu(B(1/4,\epsilon))}{\mu(B(1/4,\epsilon))}=1-\frac{1}{2^2}=\frac{3}{4}.\end{equation}
	
For $\epsilon>0$ we set
\begin{equation}\nonumber J_\epsilon=\{x\in J\,:\,f^{k}(x)\not\in B(1/4,\epsilon)\text{ for }k=0,1,2,\ldots\}.\end{equation}
	
Let $s_\epsilon={\rm dim_H}(J_\epsilon)$ and $s=\log(2)/\log(3)$ then Theorem \ref{dimension} implies that

\begin{equation}\nonumber \lim_{\epsilon\to 0}\frac{s-s_\epsilon}{\mu(B(1/4,\epsilon))}=\frac{3}{4\log(3)}.\end{equation}

\section{Spectral properties of the transfer operator}

In this section we study the spectral properties of the transfer operator.  We first fix notation which will be used for the rest of the paper.  Throughout the rest of  this paper $c$ will denote a positive and finite constant which may change in value with successive uses.  Let $A$ denote an irreducible and aperiodic $l\times l$ matrix of zeroes and ones, i.e. there exists a positive integer $d$ such that $A^{d}>0$.  We define the subshift of finite type (associated with matrix $A$) to be 

\begin{equation}\nonumber\Sigma=\{(x_n)_{n=0}^\infty\,:\,A(x_n,x_{n+1})=1,\text{ for all }n\}.\end{equation}  
	
If we equip the set $\{0,1,\ldots,l-1\}$ with the discrete topology then $\Sigma$ is compact in the corresponding Tychonov product topology.  The shift $\sigma:\Sigma\to\Sigma$ is defined by $\sigma(x)=y$, where $y_n=x_{n+1}$ for all $n$, i.e. the sequence is shifted one place to the left and the first entry deleted.  

For $\theta\in (0,1)$ we define a metric on $\Sigma$ by $d_\theta(x,y)=\theta^{m}$, where $m$ is the least positive integer (assuming that such a $m$ exists) with $x_m\neq y_m$, otherwise we set $d_\theta(x,x)=0$.  Equipped with the metric $d_\theta$, the space $(\Sigma,d_\theta)$ is complete, and moreover the topology induced by $d_\theta$ agrees with the previously mentioned Tychonov product topology.  Finally, for $x\in \Sigma$ and a positive integer $n\geq 1$ we define the cylinder of length $n$ centred on $x$ to be the set $[x]_n=[x_0,x_1,\ldots,x_{n-1}]=\{y \in\Sigma\,:\,y_i=x_i\text{ for }i=0,1,\ldots,n-1\}$.

Fix a $d_\theta$-Lipschitz continuous function $\phi:\Sigma\to\mathbb{R}$, and recall that we let $\mu$ denote its equilibrium state defined in the introduction, i.e.,

\begin{equation}\nonumber P(\phi):=\sup\left\{h_{\nu}+\int \phi d\nu\,:\,\sigma_*(\nu)=\nu,\,\nu(\Sigma)=1\right\}=h_{\mu}+\int \phi d\mu.\end{equation}

We let 
\begin{equation}\nonumber L^1(\mu):=\left\{w:\Sigma\to\mathbb{C}\,:\,w\text{ is measurable and }\int |w| d\mu<\infty\right\},\end{equation}
which equipped with the norm $\|w\|_1=\int |w| d\mu$ is a Banach space.  We now describe a particular subspace of $L^1(\mu)$ on which the transfer operator will act: for $w\in L^{1}(\mu)$, $x\in\Sigma$ and a positive integer $m$ we set 

\begin{equation}\nonumber {\rm osc}(w,m,x)={\rm esssup}\{|w(y)-w(z)|\,:\,y,z\in [x]_m\}.\end{equation}  

We introduce the semi-norm

\begin{equation}\nonumber |w|_{\theta}=\sup_{m\geq 1} \theta^{-m} \|  {\rm osc}(w,m,\cdot)\|_1.\end{equation}
	
We let 

\begin{equation}\nonumber \mathcal{B}_{\theta}=\{w\in L^{1}(\mu)\,:\,|w|_{\theta}<\infty\}.\end{equation}
	
It is worth noting if we were to take the supremum norm $\|\cdot\|_\infty$ in place of the $L^1$ norm then the space coincides with Lipschitz continuous functions (with respect to the metric $d_\theta$).  

We equip $B_{\theta}$ with the norm

\begin{equation}\nonumber \|w\|_{\theta}=|w|_{\theta}+\|w\|_1.\end{equation}

This space was first introduced by Keller \cite{Kel85}, in a more general framework, where the following result was also proved:

\begin{prop}[Keller]The space $(\mathcal{B}_{\theta},\| \cdot \|_{\theta})$ is complete.  Furthermore, the set $\{w\in\mathcal{B}_{\theta}\,:\,\|w\|_{\theta}\leq c\}$ is $L^1$-compact for any $c>0$. \label{Kellercompact}\end{prop}

We introduce the transfer operator $\mathcal{L}=\mathcal{L}_\phi:\mathcal{B}_{\theta}\rightarrow\mathcal{B}_{\theta}
$ 

\begin{equation}\nonumber(\mathcal{L} w)(x)=\sum_{\sigma(y)=x} e^{\phi(y)} w(y).\end{equation}  
	
We let $i=(i_0,i_1,\ldots,i_{k-1})$ denote an allowed string of length $k$ then we can write $(\mathcal{L}^k w)(x)=\sum_{|i|=k}e^{\phi^k(ix)}w(ix)$ where the sums is over those strings for which the concatenation $ix$ is allowed, i.e. we require $ix\in\Sigma$. 	

Another Banach space that we require is that of Lipschitz functions

\begin{equation}\nonumber \mathcal{F}_\theta=\{w:\Sigma\to\mathbb{C}\,:\,\sup_{m\geq 1}\theta^{-m}\|osc(w,m,\cdot)\|_\infty<\infty\}.\end{equation}
	
The following theorem describes the spectral properties of $\mathcal{L}$ acting on the space $\mathcal{F}_\theta$ of $d_\theta$-Lipschitz continuous functions, for a proof see \cite{ParPol90}[Theorem 2.2].

\begin{prop}[Ruelle] Let $\phi\in\mathcal{F}_\theta$ be real valued and suppose $A$ is irreducible and aperiodic.
			\begin{enumerate}
			\item There is a simple maximal positive eigenvalue $\lambda=\lambda_\phi$ of $\mathcal{L}$ with corresponding strictly positive eigenfunction $g=g_\phi\in\mathcal{F}_{\theta}$.
			\item The remainder of the spectrum of $\mathcal{L}:\mathcal{F}_{\theta}\rightarrow\mathcal{F}_{\theta}$ (excluding $\lambda>0$) is contained in a disk of radius strictly smaller that $\lambda$.
			\item There is a unique probability measure $\nu$ such that $\mathcal{L}^{*}\nu=\lambda\nu$.
			\item $\lambda^{-k} \mathcal{L}^{k} w \to g \int w d\nu$ uniformly for all $w\in\mathcal{F}_{\theta}$, where $g$ is as above and $\int g d\nu=1$.
			\end{enumerate}
			\label{ruelleopthm}
			\end{prop}	
We remark that the equilibrium state $\mu$ is absolutely continuous with respect to the eigenmeasure $\nu$, with the Radon-Nikodym being given by the eigenfunction $g$.  By scaling the operator $\mathcal{L}$, if necessary, we may assume without loss of generality that $\lambda=1$, further as $g>0$ we may assume that $\mathcal{L} 1=1$.  

Another useful property of $\mu$ is the Gibbs property (see \cite{Bow08} for further details). 	Namely, there exists a constant $c>1$ such that for any $x\in\Sigma$ and positive integer $n$ we have that 

\begin{equation}\label{gibbsoriginal} c^{-1}\leq \frac{\mu[x]_n}{e^{\phi^n(x)}}\leq c.\end{equation}
	
We now prove a result relating to the spectrum of $\mathcal{L}$ acting on $\mathcal{B}_{\theta}$, namely that it has a spectral gap.  A crucial part in this process is proving a Lasota-Yorke inequality. \footnote{The term `Lasota-Yorke' refers to the modern usage dating back to their paper \cite{LasYor73}.  Similar inequalities date back to Ionescu-Tulcea Marinescu \cite{IontulMar50} and perhaps earlier.}

\begin{lemma}There exists $c>0$ such that for any $w\in\mathcal{B}_{\theta}$ we have
\begin{equation}\nonumber |\mathcal{L}^k w|_\theta\leq c\left(\theta^k |w|_\theta+\|w\|_1\right).\end{equation}	\label{LY1}
	\end{lemma}
	
\begin{proof}Let $x,y\in\Sigma$ be such that $d_\theta(x,y)\leq \theta^m$ then 

\begin{eqnarray}\nonumber |\mathcal{L}^k w(x)-\mathcal{L}^k w(y)| & \leq & \sum_{|i|=k} |e^{\phi^k(ix)}w(ix)-e^{\phi^k(iy)}w(iy)|\\
	\nonumber & \leq & 	\sum_{|i|=k} e^{\phi^k(ix)} {\rm osc}(w,k+m,ix)+e^{\phi^k(ix)}|1-e^{\phi^k(iy)-\phi^k(ix)}||w(iy)|\\
	\nonumber & \leq & 	c\sum_{|i|=k} \left(e^{\phi^k(ix)} {\rm osc}(w,k+m,ix)+\theta^m \frac{e^{\phi^k(ix)}}{\mu[ix]_{k+m}}\int_{[ix]_{k+m}}|w|d\mu\right)\\
\nonumber &\leq & 	c\sum_{|i|=k} \left(e^{\phi^k(ix)} {\rm osc}(w,k+m,ix)+\theta^m \frac{c}{\mu[x]_{m}}\int_{[ix]_{k+m}}|w|d\mu\right).\end{eqnarray}

Where we used the Gibbs property (\ref{gibbsoriginal}) in the final line, i.e.

\begin{equation}\nonumber \frac{e^{\phi^k(ix)}}{\mu[ix]_{k+m}} \leq \frac{c}{e^{\phi^m(x)}} \leq \frac{c^2}{\mu[x]_m}.\end{equation}
	
Thus

\begin{equation}\nonumber {\rm osc}(\mathcal{L}^k w,m,x) \leq c\sum_{|i|=k} \left(e^{\phi^k(ix)} {\rm osc}(w,k+m,ix)+\theta^m \frac{c}{\mu[x]_{m}}\int_{[ix]_{k+m}}|w|d\mu\right).
 \end{equation}

Integrating with respect to $\mu$ we see that and again invoking (\ref{gibbsoriginal}) we see

\begin{equation}\nonumber \int {\rm osc}(\mathcal{L}^k w,m,x) d\mu(x) \leq c\left( \int{\rm osc}(w,k+m,x)d\mu(x)+\theta^m\|w\|_1\right),\end{equation} 
	
dividing by $\theta^m$ and taking suprema yields

\begin{equation} \nonumber |\mathcal{L} w|_\theta\leq c(\theta^k |w|_\theta+\|w\|_1).\end{equation}

Finally we see that

\begin{eqnarray}\nonumber \|\mathcal{L}^k w\|_\theta&=&|\mathcal{L}^k w|_\theta+\|\mathcal{L}^k w\|_1 \\
	\nonumber &\leq & c\theta^k(|w|_\theta+\|w\|_1) + \|w\|_1\\
	\nonumber &\leq & c(\theta^k|w|_\theta + \|w\|_1).\end{eqnarray}

	\end{proof}

\begin{lemma}The operator $\mathcal{L}:\mathcal{B}_{\theta}\to\mathcal{B}_{\theta}$ has a simple maximal eigenvalue $\lambda=1$, while the rest of the spectrum is contained in a ball of radius strictly less than $1$.\label{specgap}\end{lemma}

\begin{proof}We begin by proving that for any $w\in\mathcal{B}_\theta$ that $\mathcal{L}^k w$ converges $\int w d\mu$ in $L^1(\mu)$.  Fix $\epsilon>0$ and choose $v\in\mathcal{F}_\theta$ such that $\|v-w\|_1<\epsilon/3$, by Proposition \ref{ruelleopthm} there exists a positive integer $N$ such that $\|\mathcal{L}^n(v)-\int v d\mu\|_1<\epsilon/3$ for all $n\geq N$, in which case we see that

\begin{eqnarray}\nonumber \left\|\mathcal{L}^n(w)-\int w d\mu\right\|_1 &\leq &\| \mathcal{L}^n(w-v)\|_1+\left\|\mathcal{L}^n(v)-\int v d\mu\right\|_1+ \left\|\int v d\mu - \int w d\mu \right\|_1\\
\nonumber & \leq& 2\|v-w\|_1+\left\|\mathcal{L}^n(v)-\int v d\mu\right\|_1 <\epsilon.\end{eqnarray}

This in turn implies that for each $w\in B=\{v\in\mathcal{B}_\theta\,:\,\|v\|_\theta\leq 1\}$ that

\begin{equation}\nonumber \|\mathcal{L}^n(w)|_{\mathbb{C}^{\bot}}\|_1 = \inf_{c\in\mathbb{C}} \|\mathcal{L}^n(w)-c\|_1\to 0\,\,\text{as }n\to\infty \end{equation}
where $\mathbb{C}^{\bot}=\{w\in\mathcal{B}_\theta\,:\,\int w d\mu=0\}$.  We claim that this convergence is uniform over $B$.  To see this fix $\delta>0$ and $w\in B$ then there exists a positive integer $N=N(w)$ such that $\|\mathcal{L}^n(w)|_{\mathbb{C}^{\bot}}\|_1\leq \delta/2$ for all $n\geq N$.  By Proposition \ref{Kellercompact} $B$ is compact and so the cover $\{B_{1}(w,\delta/2)\}_{w\in B}$ has a finite subcover, say $B_1(w_1,\delta/2),B_1(w_2,\delta/2),\ldots, B_1(w_m,\delta/2)$.  In which case if $n\geq N:=\max_{i=1,2,\ldots,m} N(w_i)$ we have $\|\mathcal{L}^n_\phi (w)|_{\mathbb{C}^{\bot}}\|_1\leq\delta$ for any $w\in B$.

Finally to show the existence of a spectral gap from Proposition \ref{LY1} we observe for $w\in B$, and $n\geq N$ that

\begin{eqnarray}\nonumber \|\mathcal{L}^{2n}(w)|_{\mathbb{C}^\bot}\|_\theta &\leq & c(\theta^n|\mathcal{L}^n(w)|_{\mathbb{C}^\bot}|_\theta+\|\mathcal{L}^n(w)|_{\mathbb{C}^\bot}\|_1) \\
	\nonumber &\leq & c(\theta^{2n}|w|_{\mathbb{C}^\bot}|_\theta+\theta^n\|w|_{\mathbb{C}^\bot}\|_{1}+\|\mathcal{L}^n(w)|_{\mathbb{C}^\bot}\|_1) \\
\nonumber &\leq & c(\theta^{2n}+\theta^n+\delta).\end{eqnarray}

We may choose $n$ and $\delta$ so that $\|\mathcal{L}^{2n}(w)|_{\mathbb{C}^\bot}\|_\theta<1$ which proves that $\mathcal{L}$ has a spectral gap.\end{proof}

\subsection{Singular perturbations of the transfer operator}

We introduce a perturbation of the transfer operator $\mathcal{L}$: let $\{U_n\}_n$ be a  family of open sets, further, we require that they satisfy the following technical conditions:
\begin{enumerate}\item $\{U_n\}_n$ are nested with $\cap_{n\geq 1} U_n = \{z\}$. \label{ass1}	
	\item Each $U_n$ consists of a finite union of cylinder sets, with each cylinder having length $n$.
	\item There exists constants $c>0$, $0<\rho<1$ such that $\mu(U_n)\leq c\rho^n$ for $n\geq 1$.
	\item There a sequence $\{l_n\}_n\subset \mathbb{N}$, and constant $\kappa>0$ such that $\kappa<l_n/n\leq 1$ and  $U_n\subset [z]_{l_n}$ for all $n\geq 1$.
	
	\item If $\sigma^p(z)=z$ has prime period $p$ then $\sigma^{-p}(U_n)\cap [z_0 z_1\cdots z_{p-1}]\subseteq U_n$ for large enough $n$. \label{ass4}
\end{enumerate}

\begin{remark}We observe that that (\ref{ass4}) above is not absolutely essential for the application to conformal repellers and serves only to greatly simplify the analysis. \end{remark}

For $n\geq 1$ we define the perturbed operator $\mathcal{L}_n:\mathcal{B}_\theta\to\mathcal{B}_\theta$ by

\begin{equation}\nonumber \mathcal{L}_n(w)(x)=\mathcal{L}(\chi_{U_n^c}w)(x).\end{equation}

For a positive integer $n$ we let $\Sigma_n=\bigcap_{k\geq 0} \Sigma\setminus \sigma^{-j}(U_n)$.  By choosing $n$ large enough we can ensure that the system $(\Sigma_n,\sigma|_{\Sigma_n})$ is topologically mixing, and so the results of \cite{ColMarSch97} apply, namely we have

\begin{prop}[Collet, Mart\'{\i}nez, Schmitt]For each $n$ there exists continuous $g_n:\Sigma\to\mathbb{R}$ with $g_n>0$, and $\lambda_n>0$ such that $\mathcal{L}_n g_n =\lambda_n g_n$, moreover for any continuous $w:\Sigma\to\mathbb{C}$ we have
\begin{equation}\nonumber \|\lambda_n^{-k}\mathcal{L}_n^k w - \nu_n(w|_{\Sigma_n}) g_n\|_\infty\to 0,\end{equation}
where $\nu_n$ denotes the unique probability measure guaranteed by Proposition \ref{ruelleopthm}, i.e. $\nu_n$ satisfies ${\rm supp}(\nu_n)=\Sigma_n$ and $(\mathcal{L}_n^{*}\nu_n)(w)=\lambda_n \nu_n(w)$ for $w\in\mathcal{F}_\theta(\Sigma_n)$.\label{collopthm}
\end{prop}
 
Moreover, we may prove a Lasota-Yorke style inequality for $\mathcal{L}_n:\mathcal{B}_\theta\to\mathcal{B}_\theta$, which in conjunction with Proposition \ref{collopthm} and the methods of Lemma \ref{specgap} we can show that $g_n\in\mathcal{B}_\theta$ and that $\lambda_n$ is a simple maximal eigenvalue for $\mathcal{L}_n:\mathcal{B}_\theta\to\mathcal{B}_\theta$.

The perturbation $\mathcal{L}_n$ is singular with respect to the $\|\cdot\|_\theta$ norm, we adopt the approach of \cite{KelLiv99} and introduce a weak norm.

\begin{equation}\nonumber \|w\|_h:=|w|_h+\|w\|_1=\sup_{j\geq 0} \sup_{m\geq 1} \theta^{-m} \int_{\sigma^{-j}(U_m)} |w| d\mu+\|w\|_1.\end{equation}

Throughout this section we assume that $\theta\in (\rho, 1)$.  Our first result states that the weak norm is dominated by strong norm.

\begin{lemma}Under the assumptions above we have
\begin{equation}\nonumber \|w\|_{h} \leq c\|w\|_\theta\end{equation}
for all $w\in\mathcal{B}_\theta$.\label{norms}\end{lemma}

\begin{proof}We first relate the strong norm with the $L^{\infty}$ norm.  Let $c=\max_{i=0,1,\ldots,l-1} \mu[i]_{1}^{-1}$, then for $\mu$ almost all $x\in \Sigma$ 

\begin{eqnarray}\nonumber |w(x)| &\leq& {\rm osc}(w,1,x)+ c\int_{[x_0]_{1}} |w|d\mu \\
\nonumber &\leq &  c\left(\int_{[x_0]_{1}} {\rm osc}(w,1,y)d\mu(y)+\int_{[x_0]_{1}} |w|d\mu\right)\\ \label{supnorm} &\leq& c\| w\|_\theta.\end{eqnarray}
If $\theta \in( \rho, 1 )$ then
\begin{equation}
\nonumber |w|_h \leq  \sup_{m\geq 1} \theta^{-m}\mu(U_m) \|w\|_{\infty}\leq c\|w\|_\theta.\end{equation}
\end{proof}

\subsection{Convergence of the spectral radii}

In this section we prove a preliminary result relating to the behaviour of the spectra of the operators $\mathcal{L}_n$ acting on $\mathcal{B}_\theta$.   From Proposition \ref{collopthm} it is easy to see that for any $u\in\Sigma$ we have

\begin{equation}\label{sumpressure} P_{\Sigma_n}(\phi):=\log \lambda_n = \lim_{k\to\infty} \frac{1}{k} \log \left( \mathcal{L}_n^k 1 (u)\right).\end{equation}

\begin{prop}Under assumptions (1)-(5) we have $\lim_{n\to\infty}\lambda_n=\lambda.$\label{prelconv}\end{prop}

\begin{proof}As $U_n\subset [z]_{l_n}$, setting $\tilde{\Sigma}_n=\Sigma\setminus\cap_{k\geq 0} \sigma^{-k}[z]_{l_n}$ it is easy to see that $\tilde{\Sigma}_n\subset \Sigma_n$.  Accordingly, it suffices to show that $P_{\tilde{\Sigma_n}}(\phi)\to P(\phi)$.

As $(\Sigma,\sigma)$ is topologically mixing we may find a positive integer $d$ such that $A^d>0$.  Fix $u\in\Sigma$ and for integers $k$ and $n$ we set 
\begin{eqnarray}\nonumber B_{k}&=&\{x_0 x_1\cdots x_{k-1}\,:\, x_0 x_1 \cdots x_{k-1} u\in \Sigma\},\\
\nonumber B_{k,n}&=&\{x_0 x_1\cdots x_{k-1}\in B_k\,:\,[x_0 x_1 \cdots x_{k-1}]\cap\Sigma_n\neq\emptyset \},\end{eqnarray}
\begin{equation}\nonumber Z_{k}(\phi)=\sum_{x_0 x_1 \cdots x_{k-1}\in B_k} e^{\phi^k(x_0 x_1 \cdots x_{k-1} u)},\,\,\,\,\,  Z_{k,n}(\phi)=\sum_{x_0 x_1 \cdots x_{k-1}\in B_{k,n}} e^{\phi^k(x_0 x_1 \cdots x_{k-1} u)}.\end{equation}

It is easy to see that $\mathcal{L}^k1(u)=Z_k(\phi)$ (resp. $\mathcal{L}_n^k1(u)=Z_{k,n}(\phi)$) and so by equation (2)
we have that $P(\phi)=\lim_{k\to\infty} \frac{1}{k}\log Z_k(\phi)$ (resp. $P_{\Sigma_n}(\phi)=\lim_{k\to\infty} \frac{1}{k}\log Z_{k,n}(\phi)$).

Fix $\epsilon>0$, by equation \ref{sumpressure} there exists $a>0$ such that  $Z_k(\phi)\geq  a e^{k (P(\phi)-\epsilon)}$ for all $k\geq 1$.  In addition, as $h_{top}(\sigma)>0$ there exists $b>0$ such that $|B_k| \geq b e^{k (h_{top}(\sigma)-\epsilon)}$ for all $k\geq 1$. 

Fix large integers $k$ and $n$ such that both $b e^{k (h_{top}( \sigma)-\epsilon)}>l_n-k+1$ and $2(k+d)<l_n\epsilon$.  Observe that the string $z_0 z_1 \cdots z_{l_n-1}$ has precisely $l_n-k+1$ subwords of length $k$, accordingly the first condition on $k$ and $n$ guarantees the existence of a finite word $x\in B_k$ such that $x$ does not appear as a subword of $z_0 z_1 \cdots z_{l_n-1}$.  Fix $m\in\mathbb{N}$ and let $y^{1},y^{2},\ldots,y^{m}\in B_{l_n-2k-2d}$, we now associate with this list an unique element of $B_{m(l_n-k)}$.  Choose $s^1,s^2,\ldots, s^m, t^{1},t^{2},\ldots, t^{m}\in B_d$ so that the word $w:=y^1 s^1 x t^1 y^2 s^2 x t^2 \cdots t^{m-1} y^m s^m x t^m \in \{1,2,\ldots, l\}^{m(l_n-k)}$ is such that $t^m u\in\Sigma$, this is possible as $A^d>0$.  

It is easy to see that as $x$ is contained in any subword of length $n$, the word $z_0 z_1 \cdots z_{l_n-1}$ cannot be contained as a subword of the periodic extension of $w$.  Hence $w\in B_{m(l_n-k),l_n}$, and so 

\begin{equation}\nonumber Z_{m(l_n-k),l_n}(\phi) \geq (a e^{(l_n-2k-2d) (P(\phi)-\epsilon)})^m > (a e^{l_n(1-\epsilon)(P(\phi)-\epsilon)})^m.\end{equation}

Taking logs, dividing by $m$ and letting $m\to\infty$ yields

\begin{equation}\nonumber P_{\tilde{\Sigma}_n}(\phi)\geq \frac{\log(a)}{l_n-k}  + (1-\epsilon)\frac{l_n}{l_n-k}(P(\phi)-\epsilon).\end{equation}

Finally letting $n\to\infty$ and $\epsilon \to 0$ gives the result.
\end{proof}

\begin{remark}The proof of Proposition \ref{prelconv} is modified from \cite{Lin88} where an analogous result for topological entropy is proved.  \end{remark}

\subsection{A Uniform Lasota-Yorke inequality}

We now prove that the transfer operators $\mathcal{L}_n$ satisfy a uniform Lasota-Yorke inequality. We assume that the transfer operator $\mathcal{L}$ is normalised, i.e $\mathcal{L} 1=1$.  Iterating the perturbed operator $\mathcal{L}_n$ we see that
\begin{equation} \nonumber (\mathcal{L}_n^k w)(x)=\sum_{\sigma^{k}(y)=x} h_{n,k}(y)e^{\phi^k(y)}w(y),\end{equation}
where $h_{n,k}(x)=\prod_{j=0}^{k-1} \chi_{U_n^c}(\sigma^j x)$ and $\phi^k(y)=\sum_{j=0}^{k-1}\phi(\sigma^{j}(y))$.

\begin{lemma}For any positive integers $k,n$ we have 
\begin{equation}\nonumber \|\mathcal{L}_n^{k} \|_{h} \leq 1.\end{equation}\label{weak}\end{lemma}

\begin{proof}Let $w\in L^{1}$, then 
\begin{eqnarray} \nonumber \|\mathcal{L}_nw\|_{1} & = & \int | \mathcal{L} \chi_{U_n^{c}} w | d\mu \\
\nonumber &\leq & \int \mathcal{L} | \chi_{U_n^{c}} w | d\mu\\
\label{L1} &= & \int  | \chi_{U_n^{c}} w | d\mu \leq \|w\|_{1}.\end{eqnarray}
In addition, fixing $j\geq 0$, $m\geq 1$ we see that
\begin{eqnarray} \nonumber \theta^{-m}\int_{\sigma^{-j}(U_m)} |\mathcal{L}_n w| d\mu &\leq& \theta^{-m}\int_{\sigma^{-j}(U_m)} \mathcal{L}(|\chi_{U_n^{c}} w|) d\mu \\
\nonumber & = & \theta^{-m}\int_{\sigma^{-(j+1)}(U_m)}\chi_{U_n^{c}} |w| d\mu \\
\nonumber &\leq &  \theta^{-m}\int_{\sigma^{-(j+1)}(U_m)}|w| d\mu \leq |w|_h.\end{eqnarray}
Taking the supremum over $j$ and $m$ yields
\begin{equation}\label{hole} |\mathcal{L}_nw|_h \leq |w|_h .\end{equation}
Combining equations (\ref{L1}) and (\ref{hole}) and iterating completes the proof.
\end{proof}

\begin{lemma}There exists a constant $c>0$ such that for any positive integers $n,k$ we have 
\begin{equation}\nonumber |h_{n,k}w|_\theta \leq |w|_\theta+c\theta^{-k}\|w\|_h\end{equation}for all $w\in\mathcal{B}_\theta$.	\label{hnk}\end{lemma}
	
\begin{proof}	
	We prove this by induction, namely we prove that for any $w\in\mathcal{B}_\theta$ we have
	\begin{equation}\label{induction} |\chi_{\sigma^{-j}(U_n^c)} w|_\theta \leq |w|_\theta+\theta^{-j}\|w\|_h.\end{equation}

To show this, fix a positive integer $m$, we consider two cases, namely: $j+n\leq m$ and $m<j+n$.  If we suppose that $j+n\leq m$ then ${\rm osc}(\chi_{\sigma^{-j}(U_n^c)}w,m,x)\leq {\rm osc}(w,m,x)$ for all $x\in\Sigma$, and thus

\begin{equation}\label{1case1} \theta^{-m}\int {\rm osc}(\chi_{\sigma^{-j}(U_n^c)}w,m,x) d\mu(x) \leq \theta^{-m}\int {\rm osc}(w,m,x)d\mu(x) \leq |w|_\theta.\end{equation}

On the other hand if $m<j+n$ it is easy to see that if $[x]_m\subset \sigma^{-j}(U_n^c)$ then ${\rm osc}(\chi_{\sigma^{-j}(U_n^c)}w,m,x)={\rm osc}(w,m,x)$.  On the  other hand if $[x]_m\cap\sigma^{-j}(U_n)\neq\emptyset$ then ${\rm osc}(\chi_{\sigma^{-j}(U_n^c)}w,m,x)=\max( {\rm osc}(w,m,x), \|\chi_{[x]_m} w \|_\infty)$, in which case

\begin{eqnarray}\nonumber {\rm osc}(\chi_{\sigma^{-j}(U_n^c)}w,m,x) &=& \max({\rm osc}(w,m,x),\|\chi_{[x]_m}w\|_\infty) \\
\nonumber &\leq& {\rm osc}(w,m,x)+\frac{1}{\mu[x]_m}\int_{[x]_m}|w|d\mu.\end{eqnarray}

Which implies that

\begin{equation}
\label{case2} \theta^{-m}\int {\rm osc}(\chi_{\sigma^{-j}(U_n^c)}w,m,x) d\mu(x) \leq |w|_\theta+\theta^{-m}\int_{\{x\,:\,[x]_m\cap\sigma^{-j}(U_n)\neq\emptyset\}}|w|d\mu.\end{equation}
		
We now analyse two further subcases, if $m\leq j$ then we see that 

\begin{equation}
\label{1case21}\theta^{-m}\int_{\{x\,:\,[x]_m\cap\sigma^{-j}(U_n)\neq\emptyset\}}|w|d\mu\leq \theta^{-j}\|w\|_1.\end{equation}

If $j<m<j+n$, the fact that the open sets $\{U_n\}_n$ are nested implies that

\begin{equation}\nonumber \{x\,:\,[x]_m\cap\sigma^{-j}(U_n)\neq\emptyset\}\subset \sigma^{-j}(U_{m-j}).\end{equation}

In which case

\begin{equation}\label{1case22} \theta^{-m}\int_{\{x\,:\,[x]_m\cap\sigma^{-j}(U_n)\neq\emptyset\}}|w|d\mu\leq \theta^{-j}|w|_h.\end{equation}

If we combine equations (\ref{1case1}), (\ref{1case21}) and (\ref{1case22}) we obtain (\ref{induction}).  This completes the proof.\end{proof}

\begin{lemma}There exists a constant $c>0$ such that 
	\begin{equation}\nonumber \|\mathcal{L}^k_nw\|_\theta\leq c(\theta^k \|w\|_\theta+\|w\|_{h})\end{equation} for all $w\in\mathcal{B}_\theta$ and $n,k\geq 1$.\label{LasotaYorke}
\end{lemma}

\begin{proof}
	
Fix $x,y\in \Sigma$ and suppose $d_\theta(x,y)=\theta^m$, with $m\geq 1$, then 
\begin{eqnarray}\nonumber |(\mathcal{L}_n^k w )(x)-(\mathcal{L}_n^k w )(y) | & \leq & \sum_{|i|=k}|e^{\phi^k(ix)}h_{n,k}(ix)w(ix)-e^{\phi^k(iy)}h_{n,k}(iy)w(iy)| \\
	\nonumber &\leq & \sum_{|i|=k}e^{\phi^k(ix)}|h_{n,k}(ix)w(ix)-h_{n,k}(iy)w(iy)|\\ \nonumber && +e^{\phi^k(ix)}|1-e^{\phi^k(iy)-\phi^k(ix)}||w(iy)|\\
	\nonumber &\leq & \sum_{|i|=k} e^{\phi^k(ix)}\left[{\rm osc}(h_{n,k}w,k+m,ix)+c\cdot {\rm osc}(w,k+m,ix)\right]\\ \nonumber && + c\theta^{m}\sum_{|i|=k} \frac{e^{\phi^k(ix)}}{\mu[ix]_{k+m}}\int_{\mu[ix]_{k+m}}|w|d\mu.
\end{eqnarray}
Integrating and dividing by $\theta^m$ implies that
\begin{equation}|\mathcal{L}_n^k w|_\theta \leq c\theta^k (|w|_\theta+|h_{n,k}w|_\theta)+c\|w\|_1. \label{thetaseminorm}\end{equation}

And so from equations (\ref{L1}) and (\ref{thetaseminorm}) along with lemma \ref{hnk} we deduce that

\begin{eqnarray}\nonumber \|\mathcal{L}_n^k w\|_\theta  &=& |\mathcal{L}_n^k w|_\theta + \|\mathcal{L}_n^k w\|_1 \\
\nonumber &\leq & c\theta^k (|h_{n,k}w|_\theta+ |w|_\theta) + \|w\|_{1} \\
	\nonumber & \leq & c\theta^k |w|_\theta + c\|w\|_h\leq c\theta^k \|w\|_\theta + c\|w\|_h.\end{eqnarray}
This completes the proof.
\end{proof}

\begin{remark}The advantage of introducing the weak norm $\|\cdot\|_h$ is that it overcomes the restrictions imposed by the usual weak norm $\|\cdot\|_1$.  In particular, had we considered the usual $\|\cdot\|_1$-norm it would have imposed the condition that $0<\theta<1$ be chosen sufficiently small (leading to complications later in the proof when we also require $\rho<\theta<1$).\end{remark}

\subsection{Quasi-compactness of $\mathcal{L}_n$}

A prerequisite for proving quasi-compactness of $\mathcal{L}_n$ is that the unit ball is compact with respect to the weak norm.

\begin{prop} The set $B= \{ w\in\mathcal{B}_\theta\,:\,\|w\|_\theta\leq 1\}$ is $\| \cdot \|_h$-compact. \label{compact}\end{prop}
	
\begin{proof}Let $(f_n)_n\in B$ be any sequence.  By Proposition \ref{Kellercompact} there exists a subsequence $(f_{n_k})_k$ and $f\in B$ such that $\|f_{n_k}-f\|_{1}\rightarrow 0$.  It suffices to show that $|f_{n_k}-f|_h\rightarrow 0$. As $f,f_{n_k}\in B$ we have that $c=\sup_{k\geq 1} \|f-f_{n_k}\|_\infty <\infty$. 
Fix $\epsilon>0$ and choose a positive integer $M$ such that $\theta^{-m}\mu(U_m)\leq \epsilon/c$ for all $m>M$.  Choose a positive integer $K$ such that $\|f-f_{n_k}\|_{1}\leq \theta^M \epsilon$ for all $k\geq K$. For fixed $m,j$, then if $m>M$ we have
\begin{equation}
	\label{2case1} \theta^{-m}\int_{\sigma^{-j}(U_m)} |f-f_{n_k}|d\mu  \leq  \theta^{-m}\mu(U_m)\|f-f_{n_k}\|_\infty <\epsilon.\end{equation}
	Otherwise $m\leq M$, in which case for $k\geq K$ we have
\begin{equation}\label{2case2} \theta^{-m}\int_{\sigma^{-j}(U_m)} |f-f_{n_k}|d\mu \leq \theta^{-M}\int |f-f_{n_k}|d\mu <\epsilon.\end{equation}
	
Taking equations (\ref{2case1}) and (\ref{2case2}) together implies that $|f-f_{n_k}|_h<\epsilon$ for $k\geq K$.  This completes the proof.\end{proof}

We now prove quasi-compactness of $\mathcal{L}_n$ using a critereon of Hennion.

\begin{lemma}The essential spectral radii of the operators $\mathcal{L}_n$ is uniformly bounded by $\theta$.\label{essspec}
\end{lemma}

\begin{proof}To show that the essential spectral radius of $\mathcal{L}_n$ is bounded by $\theta$ we note that Lemmas \ref{LasotaYorke} and \ref{compact} show that the operators $\mathcal{L}_n$ satisfy the hypotheses of \cite{Hen93}[Corollary 1], namely:  
	
\begin{enumerate}
\item  $\mathcal{L}_n(\{w\in\mathcal{B}_\theta\,:\,\|w\|_\theta\leq 1\})$ is conditionally compact in $ (\mathcal{B}_\theta,\|\cdot\|_h)$. \label{qcompact1}
\item For each $k$, there exists positive real number $R_k$, $r_k$ such that $\liminf_{k\to\infty}(r_k)^{1/k}=r<\lambda_n$ for which 
\begin{equation}\nonumber \|\mathcal{L}_n^k(w)\|_\theta\leq r_k\|w\|_\theta+R_k\|w\|_h\,\,\text{ for all }w\in\mathcal{B}_\theta.\end{equation}\label{qcompact2}
\end{enumerate}	

In which case we conclude that $\mathcal{L}_n$ is quasi-compact with essential spectral radius bounded by $r$.  Condition (\ref{qcompact1}) can be deduced from Proposition \ref{Kellercompact} while condition (\ref{qcompact2}) is the uniform Lasota-Yorke inequality proved in Lemma \ref{LasotaYorke}.  Finally Proposition \ref{prelconv} implies that for any $\theta\in (0,1)$ we have that $\lambda_n>\theta$ for large $n$.\end{proof}

\subsection{Stability of the spectrum}

We introduce a so called `asymmetric operator norm' for which the operators $\mathcal{L}_n$ converge to $\mathcal{L}$ as $n\to\infty$.  For a linear operator $Q:\mathcal{B}_\theta\rightarrow \mathcal{B}_\theta$ we define 
\begin{equation}\nonumber \|| Q|\| = \sup\{ \|Qw\|_h \,:\, \|w\|_\theta\leq 1\}. \end{equation}
	
Recall the Gibbs property (\ref{gibbsoriginal}) of $\mu$, namely that there exists a constant $c>1$ such that for any $x\in\Sigma$ and positive integer $n$ we have that 

	\begin{equation} c^{-1}\leq \frac{\mu[x]_n}{e^{\phi^n(x)}}\leq c.\end{equation}

Using the above it is relatively easy to show the following proposition, which is stated without proof.

\begin{prop}[Gibbs property]There exists a constant $c>0$ such that for any positive integers $n,m$ and $j\geq n$ we have

\begin{equation}\nonumber\mu(U_n\cap\sigma^{-j}(U_m))\leq c\mu(U_n)\mu(U_m).\end{equation}	
		
\label{gibbs}\end{prop}

\begin{lemma} There exists a constant $c>0$ such that 
\begin{equation}\nonumber \|| \mathcal{L}- \mathcal{L}_n |\| \leq c(\rho\theta^{-1})^n \end{equation}
for all $n$.
\label{pert}\end{lemma}

\begin{proof}Let $w\in\mathcal{B}_\theta$ be such that $\|w\|_\theta\leq 1$ then
\begin{eqnarray}\nonumber \|(\mathcal{L}- \mathcal{L}_n)w\|_1 &=& \| \mathcal{L} \chi_{U_n} w \|_1 \\
\nonumber &\leq & \| \chi_{U_n} w \|_1 \\
\label{L1pert} &\leq & \mu(U_n) \| w\|_\infty\leq c\mu(U_n)\|w\|_\theta \leq c\mu(U_n).
\end{eqnarray}
On the other hand for fixed $m,j$ we have 
\begin{equation}\nonumber \theta^{-m}\int_{\sigma^{-j}(U_m)} |\left(\mathcal{L}-\mathcal{L}_n\right) w | d\mu \leq c\theta^{-m}\mu(\sigma^{-(j+1)}(U_m)\cap U_n )\|w\|_\theta.\end{equation}
Fix positive integers $m,j$, we study three cases, namely: 
\begin{enumerate}\item $n\leq j+1$,
\item $j+1<n<m+j+1$,
\item $m+j+1\leq n$.
\end{enumerate}
	
First we suppose that $n\leq j+1$ which implies from Proposition \ref{gibbs} that

\begin{equation} \label{3case11}\theta^{-m}\mu(\sigma^{-(j+1)}(U_m)\cap U_n )\leq c\theta^{-m}\mu(U_m)\mu(U_n)\leq c\rho^n.\end{equation}

Next, we suppose that $j+1<n<m+j+1$ then observing that the nested property of $\{U_n\}_n$ gives us $\sigma^{-(j+1)}(U_m)\cap U_n\subset \sigma^{-(j+1)}(U_m)\cap U_{j+1}$, combining this with Proposition \ref{gibbs} we see that 

	\begin{eqnarray}\nonumber \theta^{-m}\mu(\sigma^{-(j+1)}(U_m)\cap U_n) &\leq & \mu(\sigma^{-(j+1)}(U_m)\cap U_{j+1}) \\
		\nonumber &\leq & c\theta^{-m}\mu(U_m)\mu(U_{j+1}) \\
		\nonumber &\leq & c\theta^{-m}\rho^{m+j+1} \\
		\nonumber &\leq & c (\theta^{-1}\rho)^{m+j+1} \\
		\label{3case12} &\leq & c(\theta^{-1}\rho)^n.\end{eqnarray}

If $n\geq m+j+1$, in which case

\begin{equation}\label{3case13} \theta^{-m}\mu(\sigma^{-(j+1)}(U_m)\cap U_n )\leq \theta^{-m}\mu(U_n)\leq \theta^{-n}\mu(U_n)\leq c(\theta^{-1}\rho)^n.\end{equation}
	
Combining equations (\ref{3case11}),(\ref{3case12}) and (\ref{3case13}) yields
\begin{equation} \nonumber |\left(\mathcal{L}-\mathcal{L}_n\right) w |_h \leq  c(\theta^{-1}\rho)^n\|w\|_\theta.
\end{equation}

Combining this with equation (\ref{L1pert}) completes the proof.	
	
	\end{proof}

We note that Lemmas \ref{weak}, \ref{LasotaYorke}, \ref{essspec} and  \ref{pert} show that the operators $\mathcal{L}_n$ satisfy the hypotheses of \cite{KelLiv99}[Theorem 1].  We now cite a specific consequence of result.

	For $\delta>0$ and $r>\theta$ let 

	\begin{equation}\nonumber V_{\delta,r}=\{z\in\mathbb{C}\,:\,|z|\leq r\,\text{or }\text{dist}(z, {\rm spec}(\mathcal{L}))\leq \delta\}.\end{equation} 

	Then by \cite{KelLiv99}[Theorem 1] there exists $N=N(\delta,r)$ such that 

	\begin{equation}S_{\delta,r}=\sup\left\{ \|(z-\mathcal{L}_n)^{-1}\|_\theta \,:\,n\geq N,\,z\in\mathbb{C}\setminus V_{\delta,r}\right\}<\infty,\label{KL1}\end{equation} where $\|(z-\mathcal{L}_n)^{-1}\|_\theta$ denotes the operator norm of $(z-\mathcal{L}_n)^{-1}:\mathcal{B}_\theta\to\mathcal{B}_\theta$.

We use may quasi-compactness of $\mathcal{L}_n$ to write

\begin{equation}\nonumber\mathcal{L}_n=\lambda_n E_n+\Psi_n,\end{equation} where $E_n$ is a projection onto the eigenspace $\{c g_n\,:\, c\in\mathbb{C}\}$ and $E_n\Psi_n=\Psi_n E_n=0$.

	\begin{prop}There exists a positive integer $N$ and constants $c>0$ and $0<q<1$ such that for all $n\geq N$ we have  
	\begin{equation}\nonumber \|\Psi_n^{k} 1\|_{\infty} \leq c q^{k}\,\,\text{for any }k\geq 1.\end{equation} \label{resolvantbound1} \end{prop}

	\begin{proof}Fix $q\in (\theta,1)$ such that $\text{spec}(\mathcal{L})\setminus \{1\}\subset B(0,q)$.  Then by Proposition \ref{prelconv} there exists a positive integer $N$ such that for all $n\geq N$, we may write using standard operator calculus
	\begin{equation}\nonumber  \Psi_n^{k}=\frac{1}{2\pi i} \int_{|t|=q} t^{k}  (t-\mathcal{L}_n)^{-1} dt.\end{equation}

	Then from Lemma \ref{norms} and equation (\ref{KL1}) above we see that

	\begin{eqnarray}\nonumber \|\Psi^{k}_n 1\|_\infty &\leq&c \|\Psi^{k}_n 1 \|_\theta \\
	\nonumber &\leq & c \int_{|t|=q} |t|^k \|(t-\mathcal{L}_n)^{-1}\|_\theta dt \\
	\nonumber &\leq &c q^k.\end{eqnarray}
	\end{proof}

\begin{remark}This result (Proposition \ref{resolvantbound1}) is claimed in an article of Hirata \cite{Hir93}.  However, the proof presented in the article contains an error which we correct in this section.  In particular, this allows us to recover the exponential and Poisson return time estimates claimed in \cite{Hir93} for conformal expanding maps. \end{remark}
	
\begin{prop}There exists a constant $c>0$ such that for all $n$ \begin{equation}\nonumber 
	\|E_n 1\|_\infty \leq c.\end{equation}\label{resolvantbound2}\end{prop}	

\begin{proof}For $n\geq N$ write
	\begin{equation}\nonumber  E_n=\frac{1}{2\pi i} \int_{|t-1|=1-q}   (t-\mathcal{L}_n)^{-1} dt.\end{equation}

	Then from Lemma \ref{norms} and the equation (\ref{KL1}) above we see that

	\begin{eqnarray}\nonumber \|E_n 1\|_\infty &\leq&c \|E_n 1 \|_\theta \\
	\nonumber &\leq & c \int_{|t-1|=1-q}  \|(t-\mathcal{L}_n)^{-1}\|_\theta dt \\
	\nonumber &\leq &c.\end{eqnarray}
\end{proof}

\section{An asymptotic formula for $\lambda_n$}

In this section we prove the following proposition.

\begin{prop}Fix $\phi\in\mathcal{B}_\theta$, then 
\begin{equation}
\nonumber \lim_{n\rightarrow\infty} \frac{\lambda -\lambda_n}{\mu (U_n ) } = 
\begin{cases}
\lambda & \text{if }z\text{ is not periodic}.\\
 \lambda( 1-\lambda^{-p}e^{\phi^p(z)} ) & \text{if }z\text{ has prime period }p.
\end{cases}
\end{equation}\label{mainpert}
\end{prop}

We prove the proposition in the case that $\mathcal{L}$ is normalised, i.e. $\mathcal{L}1=1$, the more general statement above can be deduced by scaling the operator.  

Let $m_n$ denote the restriction of $\mu$ to $I_n$, i.e.

\begin{equation}\nonumber m_n=\frac{\mu|_{U_n}}{\mu(U_n)}.\end{equation}
	
The following four lemmas were motivated by corresponding results in \cite{Hir93}.

\begin{lemma}If $z$ is non-periodic then 
	\begin{equation}\nonumber \lim_{n\rightarrow\infty} \frac{\int E_n(\mathcal{L} \chi_{U_n})dm_n}{1-\lambda_n}=\lim_{n\rightarrow\infty}\int E_n1 d m_n =1.\end{equation}\label{nonper1}\end{lemma}
		
\begin{proof}For simplicity we put
\begin{equation}\nonumber [E_n]=\int E_n(\mathcal{L}\chi_{U_n})d m_n. \end{equation}	
	
Then, by using $\mathcal{L}\chi_{U_n}=1-\mathcal{L}_n1$,

\begin{equation}\nonumber [E_n]=(1-\lambda_n)\int E_n 1 d m_n.\end{equation}	

As $z$ is non periodic, it follows from the fact that a countable intersection of nested compact sets is non-empty that for any integer $k\geq 1$, there exists $N_k$ such that $U_{N_k}\cap\sigma^{-j}(U_{N_k})=\emptyset$ for $j=1,2,\ldots,k$.
		
Then for any $x\in \sigma^{-j} U_{N_k}$, $1\leq j \leq k$ we have that  $x\not\in U_{N_k}$.  So for any $n>N_k$ and any $x\in \sigma^{-k}U_{n}$ we see that

\begin{equation}\nonumber \chi_{U_n^c}(x) \chi_{U_n^c}(\sigma(x)) \cdots \chi_{U_n^c}(\sigma^{k-1}(x))=1.\end{equation}

So for $n>N_k$ we see that

\begin{equation}\nonumber \chi_{U_n}(x)\mathcal{L}_n^k 1(x) =\chi_{U_n}(x)\mathcal{L}^k 1(x)=\chi_{U_n}(x).\end{equation}

And so $\int\mathcal{L}_n1 d m_n=1$ for all $n>N_k$.

We now use the decomposition $\mathcal{L}_n^k=\lambda_n^k E_n+\Psi_n^k$ to see that for any $k$ and $n>N_k$ we have

\begin{eqnarray}\nonumber \left|1-\int E_n1 d m_n\right| &=& \left|(\lambda_n^k-1)\int E_n 1 d m_n + \int \Psi_n^k 1 d m_n\right| \\
	\nonumber &\leq & \left|1-\lambda_n^k\right|\|E_n 1\|_\infty +\|\Psi_n^k 1\|_\infty \\
	\nonumber &\leq & c(|1-\lambda_n^k| + q^k).\end{eqnarray}
Where Propositions \ref{resolvantbound1} and \ref{resolvantbound2} were used in the final line.  This completes the proof.		
\end{proof}

\begin{lemma}If $z$ is non-periodic then 
\begin{equation}\nonumber \lim_{n\rightarrow\infty} \frac{\int E_n (\mathcal{L} \chi_{U_n})d m_n}{\mu(U_n)}=1.\end{equation}\label{nonper2}\end{lemma}
	
	\begin{proof}
	We let $T_n(x)$ denote the first return time (assuming it exists) for $x\in U_n$, i.e.,

	\begin{equation}\nonumber T_n(x)=\inf\{i\in\mathbb{N}\,:\,\sigma^i(x)\in U_n\}\end{equation} then

\begin{eqnarray}\nonumber \int T_n d m_n & = & \sum_{i=1}^{\infty} i  m_n(T_n=i) \\
\nonumber & = &  m_n(T_n=1)+\sum_{i=2}^\infty i \int \mathcal{L}_n^{i-1}(\mathcal{L}\chi_{U_n}) d m_n \\
\nonumber & = &  m_n(T_n=1)+\int E_n(\mathcal{L}\chi_{U_n})d m_n \sum_{i=2}^\infty i \lambda_n^{i-1} + \sum_{i=2}^{\infty} i \int \Psi_n^{i-1}(\mathcal{L}_n\chi_{U_n}) d m_n \\
\nonumber &=&  m_n(T_n=1)+\int E_n(\mathcal{L}\chi_{U_n})d m_n \left(\frac{1}{(1-\lambda_n)^2}-1\right) + \sum_{i=2}^{\infty} \int \Psi_n^{i-1}1 d m_n \\
\end{eqnarray}

But by Kac's theorem $\int T_n d m_n=\frac{1}{\mu(U_n)}$ and thus 
\begin{eqnarray}\nonumber \frac{\int E_n (\mathcal{L} \chi_{U_n})d m_n}{\mu(U_n)} &=& \underbrace{\left(\frac{\int E_n (\mathcal{L} \chi_{U_n})d\mu_n}{1-\lambda_n}\right)^2}_{\to 1} \\
	\nonumber &+& \underbrace{\int E_n (\mathcal{L} \chi_{U_n})d m_n}_{\to 0}\underbrace{\left( m_n(T_n=1)-\int E_n(\mathcal{L}\chi_{U_n}) d m_n+\sum_{k=1}^\infty \Psi_n^k 1 d m_n\right)}_{=O(1)}.\end{eqnarray}
	
This completes the proof.\end{proof}
	
\begin{lemma}If $z$ has prime period $p$, then
	\begin{equation}\nonumber  \lim_{n\rightarrow\infty} \frac{\int E_n(\mathcal{L} \chi_{U_n})d m_n}{1-\lambda_n}=\lim_{n\rightarrow\infty}\int E_n1 d m_n =1-e^{\phi^p(z)}.\end{equation}\label{per1}\end{lemma}

\begin{proof}Fix a large positive integer $m$ and set $k=pm$. We have that for large $n$ that 

\begin{eqnarray}\nonumber \chi_{U_n}(x)-\chi_{U_n}(x)\mathcal{L}_{n}^k 1(x) &=& \chi_{U_n}(x)\sum_{\sigma^k (y)=x} \chi_{\cup_{j=0}^{k-1}\sigma^{-j}(U_n)}(y)e^{\phi^k(y)} \\
\nonumber &=&\chi_{U_n}(x)\sum_{\sigma^k(y)=x} \chi_{\sigma^{k-p}[z_0 z_1\cdots z_{p-1}]}(y) e^{\phi^k(y)} \\
\nonumber &=& \chi_{U_n}(x)\sum_{\sigma^{pm}(y)=x}e^{\phi^{pm}(y)} \chi_{ \sigma^{-p(m-1)}[z_0 z_1 \cdots z_{p-1}]} (y) \\
\nonumber & = & \chi_{U_n}(x) \mathcal{L}^{pm}(\chi_{[z_0 z_1 \cdots z_{p-1}]}\circ\sigma^{p(m-1)} )(x)  \\
\nonumber &=& \chi_{U_n}(x) \mathcal{L}^{p}(\chi_{[z_0 z_1 \cdots z_{p-1}]}  ) (x)\end{eqnarray} where assumption (5) on the family $\{U_n\}_n$ was utilised in the second line.  Hence

\begin{eqnarray} \nonumber \left|1-e^{\phi^p(z)} - \int \mathcal{L}_n^k 1 d m_n\right| &\leq & \left|\int \mathcal{L}^{p}(\chi_{[z_0 z_1 \cdots z_{p-1}]})(x)-e^{\phi^p(z)}  d m_n(x)\right| \\
	\nonumber &\leq & \sup_{y\in [z]_{l_n+p}} |\phi^p(y)-\phi^p(z)| \\
	\nonumber &\leq &  \frac{|\phi |_{\theta,\infty} }{1-\theta} {\rm diam}(U_n) \to 0 \,(n\to\infty),\end{eqnarray}

where $|\cdot|_{\theta,\infty}$ denotes the usual H\"older semi-norm.

Hence any $k=pm$,

\begin{equation}\nonumber \lim_{n\rightarrow\infty} \int \mathcal{L}_n^k 1 d m_n = 1-e^{\phi^p(z)}.\end{equation}
	
On the other hand, by lemma \ref{nonper1}, for large $n$

\begin{equation} \nonumber \left| \int \mathcal{L}_n^k 1 d m_n - \lambda_n^k \int E_n d m_n \right| = \left| \int \Psi_n^k 1 d m_n \right| \leq \|\Psi_n^k \|_\infty \leq c q^k.\end{equation} 
	
We fixed $k=pm$ and $\lambda_n\to 1$ as $n\to\infty$, hence

\begin{equation}\nonumber \lim_{n\to\infty}\int E_n1 d m_n = 1-e^{\phi^p(z)}.\end{equation}
	\end{proof}

\begin{lemma}If $z$ has prime period $p$, then
		\begin{equation}\nonumber \lim_{n\rightarrow\infty} \frac{\int E_n(\mathcal{L} \chi_{U_n})d m_n}{\mu(U_n)}=(1-e^{\phi^p(z)})^2.\end{equation}\label{per2}\end{lemma}

\begin{proof}The proof of this is a combination of the methods from Lemma \ref{nonper2} and the result of Lemma \ref{per1}.\end{proof}

Combining Lemmas \ref{nonper1}, \ref{nonper2}, \ref{per1} and \ref{per2} proves Proposition \ref{mainpert}.

\section{Escape rates for Gibbs Measures}

In this section we prove the analogue of Theorem \ref{escratesexp} in the setting of a topologically mixing subshift of finite type, namely we prove:  

\begin{thm}Suppose that $\{U_n\}_n$ satisfy assumptions (\ref{ass1})-(\ref{ass4}). Let $\phi:\Sigma\to\mathbb{R}$ be H\"older continuous and let $\mu$ denote the associated equilibrium state, then
\begin{equation}
\nonumber
\lim_{n\rightarrow\infty} \frac{r_\mu (U_n) }{\mu(U_{n}) }=
\begin{cases} 1 &\text{if }z \text{ is not periodic} \\
1-e^{\phi^p(z)-p P(\phi)}  & \text{if }z\text{ has prime period p}\end{cases}\end{equation}
where $\phi^p(z)=\phi(z)+\phi(\sigma(z))+\cdots+\phi(\sigma^{p-1}(z))$\label{escrates}.\end{thm}

It is well known that the escape rate $r_\mu(U_n)$ is related to the spectral radius $\lambda_n$ and we include the proof of the following proposition only for completeness. 	

\begin{prop}\begin{equation}\nonumber r_\mu(U_n)=-\log(\lambda_n).\end{equation}\label{rmu}\end{prop}

\begin{proof}We can write
\begin{eqnarray}\nonumber \mu\{x\in \Sigma\,:\,\sigma^i(x) \not\in U_n\, ,0\leq i\leq k-1\} &=& \int \left(\prod_{i=0}^{k-1} \chi_{U_n^c}(\sigma^i x)\right) d\mu(x)\\
	\nonumber & =& \int \mathcal{L}^k \left(\prod_{i=0}^{k-1} \chi_{U_n^c}(\sigma^i x)\right) d\mu(x) \\
	\nonumber &=& \int \mathcal{L}_n^k 1(x)d\mu(x) \\
	\nonumber &=& \lambda_n^k \int E_n1 d\mu + \int \Psi_n^k 1d\mu.\end{eqnarray}
Using Propositions \ref{resolvantbound1} and \ref{resolvantbound2} we see that

\begin{equation}\nonumber r_\mu(U_n)=\lim_{k\to\infty} -\frac{1}{k}\log \mu\{x\in \Sigma\,:\,\sigma^i(x) \not\in U_n\, ,0\leq i\leq k-1\} = -\log(\lambda_n) .\end{equation}	
	
	\end{proof}
	
We now prove Theorem \ref{escrates}.

\begin{proof}We assume that without loss of generality that $P(\phi)=0$.  In which case we see that from Lemma \ref{rmu} that 
	
\begin{eqnarray}\nonumber \frac{r_\mu(U_n)}{\mu(U_n)} &=& \frac{-\log(\lambda_n)}{\mu(U_n)} \\
\nonumber & = & \frac{\log(\lambda)-\log(\lambda_n)}{\mu(U_n)}\\
\nonumber &=& \frac{\lambda-\lambda_n}{\mu(U_n)}\frac{\log(\lambda)-\log(\lambda_n)}{\lambda-\lambda_n}.\end{eqnarray}
The result now follows from Proposition \ref{mainpert}.\end{proof}

We also can obtain results relating to the convergence of topological pressure.

\begin{thm}Suppose that $\{U_n\}_n$ satisfy assumptions (\ref{ass1})-(\ref{ass4}). Let $\phi:\Sigma\to\mathbb{R}$ be H\"older continuous and let $\mu$ denote the associated equilibrium state, then
\begin{equation}
\nonumber
\lim_{n\rightarrow\infty} \frac{P(\phi)-P_{\Sigma_n}(\phi)}{\mu(U_{n}) }=
		\begin{cases} 1 &\text{if }z\text{ is not periodic} \\
		1-e^{\phi^p(z)-p P(\phi)}  & \text{if }z\text{ has prime period p.}\end{cases}\end{equation}
		\label{toppres}\end{thm}

\begin{proof}Using $\lambda=e^{P(\phi)}$ we see that 

\begin{equation}
\label{rewrite} \frac{P(\phi)-P_{\Sigma_n}(\phi)}{\mu(U_n)} = \frac{P(\phi)-P_{\Sigma_n}(\phi)}{e^{P(\phi)}-e^{P_{\Sigma_n}(\phi)} }\frac{\lambda-\lambda_n}{\mu(U_n)}.
		\end{equation}

Observing that $\lim_{n\to\infty} \frac{P(\phi)-P_{\Sigma_n}(\phi)}{e^{P(\phi)}-e^{P_{\Sigma_n}(\phi)} } = e^{-P(\phi)}$, and combining this, (\ref{rewrite}) and Proposition \ref{mainpert} completes the proof.\end{proof}

An immediate corollary is the following:

\begin{cor}Let $\mu$ denote the measure of maximal entropy (i.e. the Parry measure \cite{Par64}), then 
	\begin{equation}
	\nonumber
	\lim_{n\rightarrow\infty} \frac{h_{top}(\sigma)-h_{top}(\sigma|_{\Sigma_n})) }{\mu(U_{n}) }=
			\begin{cases} 1 &\text{if }z\text{ is not periodic} \\
			1-e^{-p h_{top}(\sigma)}  & \text{if }z\text{ has prime period p.}\end{cases}\end{equation}
			\label{topent}
\end{cor}

\begin{remark}The rate of convergence of topological entropy of the restriction of the shift to these sets was studied by Lind \cite{Lind89} who proved, in the case that the $U_n$ consisted of a single cylinder of length $n$, i.e. $U_n=[z]_n$, the existence of a constant $c>1$ such that

\begin{equation}\nonumber 1/c\leq \frac{h_{top}(\sigma)-h_{top}(\sigma|_{\Sigma_n})}{\mu(U_n)}\leq c\,\,\,\text{ for all }n.\end{equation}\end{remark}

\section{Proof of Theorem \ref{escratesexp}}

In this section we prove Theorem \ref{escratesexp}.  Let $\mathcal{M}$ be a Riemannian manifold and $f:\mathcal{M}\to\mathcal{M}$ a $C^{1}$-map.  Let $J$ be a compact subset of $\mathcal{M}$ such that $f(J)=J$.  We say that the pair $(J,f)$ is a conformal repeller if
\begin{enumerate}
\item $f|_{J}$ is a conformal map.
\item there exists $c>0$ and $\lambda>1$ such that $\|df^{n}_x v\|\geq c\lambda^n\|v\|$ for all $x\in J$, $v\in T_x\mathcal{M}$, and $n\geq 1$. 	
\item $f$ is topologically mixing on $J$.
\item $J$ is maximal, i.e. there exists an open neighbourhood $V\supset J$ such that
	\begin{equation}\nonumber J=\{x\in V\,:\,f^{n}(x)\in V\,\,\text{for all }n\geq 0\}.\end{equation}
\end{enumerate}

Let $\phi:J\to\mathbb{R}$ be $\alpha$-H\"older and let $\mu$ denote the associated equilibrium state.  For an open set $U\subset J$ we let $r_{\mu}(U)$ denote the escape rate of $\mu$ through $U$.  

It is well known that an expanding map has a finite Markov partition $\{R_1,R_2,\ldots,R_l\}$, and that there exists a continuous semi-conjugacy $\pi:\Sigma\to J$ where $\Sigma$ is a subshift of finite type on $l$ symbols.  By choosing $\lambda^{-\alpha}<\theta<1$ and considering $\Sigma$ equipped with the metric $d_\theta$ it can be seen that the map $\tilde{\phi}=\phi\circ\pi:\Sigma\to\mathbb{R}$ is $d_\theta$-Lipshitz, and so $\tilde{\phi}\in\mathcal{B}_\theta$.

We state without proof the following result of Bowen \cite{Bow70}.

\begin{prop}[Bowen]There exists a positive integer $d$ such that 
the cardinality of $\pi^{-1}(x)$ is at most $d$,  for all $x\in J$. \label{Bowen1}  \end{prop}

This proposition was used to prove the following corollary:

\begin{cor}[Bowen]$x\in\Sigma$ is periodic if and only if $\pi(x)\in J$ is periodic.\label{Bowen2}\end{cor} 

We also require the following technical lemmas.

\begin{lemma}For any periodic point $z\in J$ there exists a Markov partition $\{R_1,R_2,\ldots,R_m\}$ such that $z\in\bigcup_{i=1}^m \text{int}(R_i)$.\label{markovperiodic}\end{lemma}

\begin{proof}This follows easily from the standard construction of Markov partitions (using shadowing), for example see \cite{Zin00}. \end{proof}

\begin{lemma}There exists constants $s, c_1>0$ such that $\mu(B(z,\epsilon))\leq c_1\epsilon^s$ for all $\epsilon>0$.\end{lemma}
	
\begin{proof}Let $\tilde{\phi}:\Sigma\to\mathbb{R}$ be defined by $\tilde{\phi}(x)=\phi(\pi(x))$, and denote the associated equilibrium state by $\tilde{\mu}$, then $\mu=\pi^*(\tilde{\mu})$.

For $\epsilon>0$, let $\mathcal{U}_\epsilon$ denote the Moran cover associated with the Markov partition $\{R_1,R_2,\ldots,R_m\}$ (see \cite[pg. 200]{Pes97}).  Then for $z\in J$ we choose elements $U_1,U_2,\ldots,U_k\in\mathcal{U}_\epsilon$ which intersect $B(z,\epsilon)$.  A basic property of Moran covers is that:
\begin{enumerate}
\item $U_i=\pi[z^{i}_0 z^{i}_1 \cdots z_{n_(z^{i})}^i]$, where $z^{i}\in\Sigma$.
\item ${\rm diam}(U_i)\leq \epsilon < {\rm diam}\left(\pi[z^{i}_0 z^{i}_1 \cdots z_{n(z^{i})-1}^i]\right)$.
\item $k\leq K$, where $K$ is independent of both $z$ and $\epsilon$. 
\end{enumerate}
In which case it suffices to show that $\mu(U_i)\leq c\epsilon^s$ for some constant $c>0$.  To see this we observe a basic property of Gibbs measures is that for any $x\in\Sigma$ there exists $c>0$ and $\gamma\in (0,1)$ such that $\tilde{\mu}[x]_n \leq c\gamma^n$ for $n=1,2,\ldots$.  In addition to $f\in C^{1+\alpha}$ and conformal we have that $c\lambda^{-n(z^i)}\leq \epsilon$ for any $\epsilon>0$.  In which case we see that 
\begin{equation}\nonumber \mu(B(z,\epsilon))\leq \sum_{i=1}^k \mu(U_i)=\sum_{i=1}^k \tilde{\mu}[z^{i}_0 z^{i}_1 \cdots z_{n_(z^{i})}^i] \leq K c^{1+\log(\gamma)/\log(\lambda)}  \epsilon^{-\log(\gamma)/\log(\lambda)}.\end{equation} 

\end{proof}
Next we require the so called ``$D$-annular decay property'', that is there exists a constant $c_2>0$ such that for all $x\in J$ $\epsilon>0$ and $0<\delta<1$ we have that

\begin{equation}\label{Ddecay}\mu(B(x,\epsilon)\setminus B(x,(1-\delta)\epsilon))\leq c_2\delta^D\mu(B(x,\epsilon)).\end{equation}

A related condition is the ``doubling'' or ``Federer'' property, namely there exists a constant $K>1$ such that for all $x\in J$ and $\epsilon>0$ we have

\begin{equation}\nonumber \mu(B(x,2\epsilon)) \leq K \mu(B(x,\epsilon)).\end{equation}
	
Evidently, a measure that satisfies the $D$-annular decay property also satisfies the doubling property.  The converse was shown by Buckley in \cite{Buc99}[Cor 2.2].  In the context of an equilibrium state $\mu$ supported on a conformal repeller Pesin and Weiss \cite{PesWei97} showed that $\mu$ satisfies the doubling property.  We collect these two results in the following proposition:	

\begin{prop}There exists a $D$ such that $\mu$ satisfies the $D$-annular decay property.\end{prop}

We now prove Theorem \ref{escratesexp}

\begin{proof}We first prove the result if $z \in J$ is not periodic, we first observe that the map $\pi$ that a consequence of Proposition \ref{Bowen1} we have that $\pi^{-1}\{z\}=\{ z^1,z^2,\ldots,z^r\}$, further Corollary \ref{Bowen2} implies that each $z^i$ is non-periodic.  

	Hence to show Theorem \ref{escratesexp} it suffices to show that $\lim_{\epsilon\to 0}\frac{r_{\tilde{\mu}}(\pi^{-1}(B(z,\epsilon)))}{\tilde{\mu}(\pi^{-1}(B(z,\epsilon)))}=1$.  First, we observe that Theorem \ref{escrates} may be modified to accommodate multiple non-periodic points appearing in the intersection, this modification is trivial and we therefore omit the proof.  For non-periodic points the new hypotheses become: 

	\begin{enumerate}\item Let $\{V_{n}\}$ be a family of nested sets with each $V_n$ being a finite union of cylinders.  Suppose further that $\bigcap_{n\geq 1}V_n$ consists of finitely many non-periodic points $\{ z^1,z^2,\ldots,z^r\}$. \label{item1}
	\item There exists constants $c>0$ and $0<\rho<1$ such that $\tilde{\mu}(V_n)\leq c\rho^{k_n}$ for all $n\geq 1$, here $k_n$ denotes the maximum length of a cylinder in $V_n$. \label{item2}
	\item There exists a sequence $(l_n)_n$, and constant $\kappa>0$ such that $\kappa<l_n/k_n\geq 1$ and $V_n\subset \cup_{i=1}^{r}[z^i]_{l_n}$ for all $n\geq 1$. \label{item3}\end{enumerate}

	If the sets $\{V_n\}_n$ satisfy these hypotheses then we conclude 

	\begin{equation}\nonumber \lim_{n\to\infty}\frac{r_{\tilde{\mu}}(V_n)}{\tilde{\mu}(V_n)}=1.\end{equation}
	
	We first prove the theorem for the case that $z\in J$ is non-periodic.  For $\epsilon>0$ and a positive integer $k$ we set 
	\begin{equation}\nonumber U_{k,\epsilon}=\left\{U\in\bigvee_{i=0}^{k-1}f^{-i}\mathcal{R}\,:\,U\cap B(z,\epsilon)\neq\emptyset\right\}.\end{equation} 
	
We observe that due to $f$ being uniformly expanding, there exist constants $c_3>0$ and $0<\rho<1$ such that 
\begin{equation}\nonumber {\rm diam}(U)\leq c_3\rho^k\end{equation}
for any $U\in U_{k,\epsilon}$.  

Let $\delta_k=\frac{c_3\rho^k}{\epsilon+c_3\rho^k}$, in which case it is easy to see that 

\begin{equation}\nonumber \bigcup_{U\in U_{k,\epsilon}}U\subset B(z,\epsilon+c_3\rho^k)=B(z,(1-\delta_k)^{-1}\epsilon).\end{equation}

Fix $\eta>0$ small and choose $k=k(\epsilon,\eta)$ such that $\rho^k\leq \frac{\epsilon}{c_3((c_2\eta^{-1})^{1/D}-1)} < \rho^{k-1}$, in which case we see

\begin{eqnarray}\label{delta1} (1-\eta)\mu(\cup_{U\in U_{k,\epsilon}}U) &\leq & (1-c_2\delta_k^D)\mu(\cup_{U\in U_{k,\epsilon}}U) \\
 \nonumber & \leq & (1-c_2\delta_k^D)\mu(B(z,(1-\delta_k)^{-1}\epsilon)) \\
\nonumber & \leq & \mu(B(z,\epsilon)).\end{eqnarray}

Where the $D$-annular decay property was used on the final line.  Now let $\{\epsilon_n\}_n$ be any monotonic sequence with $\epsilon_n\to 0$ and set 

\begin{equation}\nonumber U_n=\bigcup_{U\in U_{k(\epsilon_n,\eta),\epsilon_n}} U.\end{equation}

Observing that $U_n$ is a finite union of $k_n:=k(\epsilon_n,\eta)$'th level refinement of the markov partition, there exists $V_n\subset\Sigma$, a finite union of cylinders of length $k_n$ such that $\pi(V_n)=U_n$.  

We claim that $V_n$ satisfies the hypotheses of the modified Theorem \ref{escrates}.  Clearly the $V_n$ are nested (\ref{item1}), so it suffices to show that $\tilde{\mu}(V_n)$ decays exponentially in $n$.  To see this we observe that
\begin{eqnarray}\nonumber \tilde{\mu}(V_n) &=& \mu(U_n) \\
 \nonumber &\leq & (1-\eta)^{-1} \mu(B(z,\epsilon_n) \\
\nonumber & \leq & c_1 \epsilon_n^s \leq c_1(c_3((c_2\eta^{-1})^{1/D}-1))^s \rho^{s(k_n-1)}.
\end{eqnarray}
And thus, we see that $\tilde{\mu}(V_n)$ decreases exponentially in $k_n$, which proves (\ref{item2}).  

As $f$ is conformal and $z\in [z^{i}]_l$ for all $i$ and $l$ there exists a constant $c_4>0$ and $0<\varrho<1$ such that for any $i\in \{1,2,\ldots,r\}$ and $l\in\mathbb{N}$ we have that $c_4^{-1}\leq{\rm diam}(\pi[z^i]_l)/\varrho^l$.  Let $l_n$ be the minimum such $l$ such that $c_4^{-1}\varrho^l\geq 2\epsilon_n$.  It is easy to see that for such a choice of $l$ we have that $V_n\subset \cup_{i=1}^{r}[z^i]_{l_n}$.  In addition, we have that $l_n >c_5 k_n$ for some constant $c_5>0$, this proves (\ref{item3}).  Thus we deduce from the modified Theorem \ref{escrates}

\begin{equation}\nonumber \lim_{n\to\infty}\frac{r_{\tilde{\mu}}(V_n)}{\tilde{\mu}(V_n)}=1.\end{equation}

And so by monotonicity of escape rates and equation (\ref{delta1}) we see that

\begin{equation}\label{upperbound} \limsup_{n\to \infty}\frac{r_{\mu}(B(z,\epsilon_n))}{\mu(B(z,\epsilon_n))}\leq (1-\eta)^{-1}\limsup_{n\to\infty} \frac{r_{\mu}(U_n)}{\mu(U_n)} = (1-\eta)^{-1}\limsup_{n\to \infty}\frac{r_{\tilde{\mu}}(V_n)}{\tilde{\mu}(V_n)} =(1-\eta)^{-1}.\end{equation}

Similarly, using the same method we may obtain a lower bound, which in conjunction with equation (\ref{upperbound}), gives 
\begin{equation}\nonumber \lim_{n\to \infty}\frac{r_\mu(B(z,\epsilon_n))}{\mu(B(z,\epsilon_n))}=1.\end{equation}

We now turn our attention to the case where $z$ is periodic.  By Lemma \ref{markovperiodic} we may assume that $\pi^{-1}(z)$ consists of a single point of prime period $p$ say $\pi(z^\prime)=z$.  

As before we approximate $B(z,\epsilon)$ from outside using elements of $\bigvee_{i=0}^{k-1}f^{-i}\mathcal{R}$, which may be thought of as cylinders of length $k$ in a subshift of finite type.  Recall the hypotheses for Theorem \ref{escrates}:

\begin{enumerate}\item Let $\{V_{n}\}$ be a family of nested sets with each $V_n$ being a finite union of cylinders.  Suppose further that $\bigcap_{n\geq 1}V_n=\{z^\prime\}$, where $z^\prime$ has prime period $p$.
\item There exists constants $c>0$ and $0<\rho<1$ such that $\tilde{\mu}(V_n)\leq c\rho^{k_n}$ for $n=1,2,\ldots$, here $k_n$ denotes the maximum length of a cylinder in $V_n$.
\item For each $n\geq 1$ we have that $\sigma^{-p}(V_n)\cap[z^\prime_0 z^\prime_1 \cdots z^\prime_{p-1}]\subset V_n$. \end{enumerate}   

In which case we deduce from Theorem \ref{escrates} that \begin{equation}\nonumber \lim_{n\to\infty}\frac{r_{\tilde{\mu}}(V_n)}{\tilde{\mu}(V_n)}=1-e^{\tilde{\phi}^p(z^\prime)}.\end{equation} 

We first approximate $B(z,\epsilon_n)$ from outside using the same method employed previously.  For $\eta>0$ we obtain $U_n\supset B(z,\epsilon_n)$ nested, each being a finite union of elements from $\bigvee_{i=0}^{k(n)-1}f^{-i}\mathcal{R}$ for some $k$, with the property that $\mu(U_n)\leq (1-\eta)^{-1}\mu(B(z,\epsilon_n))$.  As before, we may find a $V_n\subset \Sigma$ which is a finite union of cylinders.  It is easy to see that $V_n$ satisfy conditions (1) and (2).  To see (3) we observe that for $\epsilon_n$ small expansivity of $f$ and 
the fact that $z$ has prime period $p$ yields $f^{-p}(B(z,\epsilon_n))\cap \pi[z_0^\prime,z_1^\prime,\ldots,z_{p-1}^\prime]\subset B(z,\epsilon_n)$.  A simple argument extends this to approximations of balls centred on $z$.  Using monotonicity of escape rates together with the conclusions of Theorem \ref{escrates} yields

	\begin{eqnarray}\nonumber \limsup_{n\to \infty}\frac{r_{\mu}(B(z,\epsilon_n))}{\mu(B(z,\epsilon_n))} &\leq & (1-\eta)^{-1}\limsup_{n\to \infty} \frac{r_{\mu}(U_n)}{\mu(U_n)}\\ 
	\nonumber &=& (1-\eta)^{-1}\limsup_{n\to \infty}\frac{r_{\tilde{\mu}}(V_n)}{\tilde{\mu}(V_n)} \\\nonumber &=& (1-\eta)^{-1}(1-e^{\tilde{\phi}^p}(z^\prime))\\
	\label{upperbound1} &=& (1-\eta)^{-1}(1-e^{\phi^p(z)}).\end{eqnarray}

Similarly, using the same method we may obtain a lower bound, which in conjunction with equation (\ref{upperbound1}), we see that 
	\begin{equation}\nonumber \lim_{n\to \infty}\frac{r_\mu(B(z,\epsilon_n))}{\mu(B(z,\epsilon_n))}=1-e^{\phi^p(z)}.\end{equation}\end{proof}

\section{Proof of Theorem \ref{dimension}}

In this section we study the asymptotic behaviour of the Hausdorff dimension of the non-trapped set.  Let $f:J\to J$ be a conformal repeller as defined in the previous section, we make the further assumption that $f\in C^{1+\alpha}(J)$ for some $\alpha>0$.  Fix $z\in J$, for $\epsilon>0$ we define

\begin{equation}\nonumber J_{\epsilon}=\{x\in J\,:\,f^{k}(x)\not\in B(z,\epsilon), \text{for all }k\geq 0\},\end{equation}
i.e. all points whose orbits are $\epsilon$-bounded away from $z$.

Let $\mu$ denote the equilibrium state related to the potential $\psi=-s\log|f^\prime|$, where $s={\rm dim}_H(J)$.  As before we may study the escape rate $r_{\mu}(B(z,\epsilon))$ of $\mu$ through $B(z,\epsilon)$ and it's associated asymptotic, i.e.

\begin{equation}\nonumber d_\phi(z):=\lim_{\epsilon\to 0}\frac{r_{\mu}(B(z,\epsilon))}{\mu(B(z,\epsilon))}.\end{equation}
	
The method of proof is as follows:  in a similar vein to the proof of Theorem \ref{escratesexp}, we first prove the result where the hole consists of a finite union of refinements of the Markov partition, then extend it to the case of geometric balls via an approximation argument. 

 Let $\mathcal{R}=\{R_1,R_2,\ldots,R_m\}$ denote a Markov partition for the conformal repeller $J$, this induces a semi-conjugacy $\pi$ between a subshift of finite type $(\Sigma,\sigma)$ and the conformal repeller $(J,f)$.   Let $I_n\in\bigvee_{i=0}^{n-1}f^{-j}\mathcal{R}$ be a nested family such that $\cap_{n\geq 0} I_n=\{z\}$.  We let $J_n$ denote the set of points in $J$ which do not fall down the hole $I_n$, i.e.

\begin{equation}\nonumber J_n=\{x\in J\,:\,f^{k}(x)\not\in I_n, \text{for all }k\geq 0\}.\end{equation}
 
 Let $s_\epsilon$ denote the Hausdorff dimension of the set $J_\epsilon$. 

 \begin{prop}Under the assumptions above 
 \begin{equation}\nonumber \lim_{n\to\infty}\frac{s-s_n}{\mu(I_n)}=\frac{d_\phi(z)}{\int \log |f^\prime| d\mu}\end{equation}\label{dimensiondiscrete}
 \end{prop}

A cruicial ingredient to the proof of Proposition \ref{dimensiondiscrete} is the following result of Ruelle \cite{Rue82}.

\begin{prop}[Ruelle]Let $s\geq 0$ be the unique real number for which $P(-s\log|f^{'}|)=0$, then $\text{dim}_H(J)=s$.
	\end{prop}
 
Let $\tilde{\phi}(x):=-\log|f^\prime(\pi(x))|$ it is easy to see that the semi-conjugacy $\pi$ being one-one on a set of full measure for all equilibrium states for H\"older potentials implies that the Hausdorff dimension of $J$ is the unique real number $s$ for which $P(s\tilde{\phi})=0$.  As similar argument shows that the ${\rm dim_H}(J_n)=s_n$ where $s_n$ is the unique real number satisfying $P_{\Sigma_n}(s_n\tilde{\phi})=0$.  We may therefore translate the problem into the language of subshifts of finite type.  As the family $\{I_n\}$ is nested there exists a point $z^\prime\in\Sigma$ such that $\pi[z^\prime]=I_n$. Accordingly, if we set

\begin{equation}\nonumber \Sigma_n=\{x\in\Sigma\,:\,\sigma^k(x)\not\in [z^\prime]_n\text{ for }k=0,1,2,\ldots\}\end{equation}
then $\pi(\Sigma_n)=J_n$.  Let $\tilde{\phi}(x)=-\log|f^\prime(\pi(x))|$, then it is easy to see that the semi-conjugacy $\pi$ being one-one on a set of full measure for all equilibrium states for H\"older potentials implies that the Hausdorff dimension of $J$ is the unique real number $s$ for which $P(s\tilde{\phi})=0$.  As similar argument shows that the ${\rm dim_H}(J_n)=s_n$ where $s_n$ is the unique real number satisfying $P_{\Sigma_n}(s_n\tilde{\phi})=0$.  We therefore may prove the result in the setting of subshifts of finite type.

For $t\geq 0$ we let $\mathcal{L}_t:\mathcal{B}_\theta\to\mathcal{B}_\theta$ denote the transfer operator associated with the potential $t\tilde{\phi}$, i.e.,

\begin{equation}\nonumber (\mathcal{L}_t w)(x)=\sum_{\sigma(y)=x} \frac{w(y)}{|f^\prime(\pi(y))|^t},\end{equation}
analogously we define the perturbed transfer operator $\mathcal{L}_{t,n}:\mathcal{B}_\theta\to\mathcal{B}_\theta$ to be $(\mathcal{L}_{t,n}w)(x)=(\mathcal{L}_{t}\chi_{[z^\prime]_n^c}w)(x)$.  We let $g_t$ (resp. $g_{t,n}$) and $\nu_t$ (resp. $\nu_{t,n}$) denote the eigenfunction and eigenmeasures guaranteed by Proposition \ref{ruelleopthm} applied to $\mathcal{L}_t$ (resp. $\mathcal{L}_{t,n}$ ).  We shall assume without loss of generality that $\int g_t d\nu_t=\int g_{t,n} d\nu_{t,n}=1$ for all $t\geq 0$ and $n\geq 1$.  The associated equilibrium states will be denoted by $\mu_t$ and $\mu_{t,n}$, observing that one can show that $d\mu_t=g_t d\nu_t$ (resp. $d\mu_{t,n}=g_{t,n} d\nu_{t,n}$).  We proved earlier that both $\mathcal{L}_t$ and $\mathcal{L}_{t,n}$ have spectral gaps, we denote their maximal eigenvalues by $\lambda_t$ and $\lambda_{t,n}$ respectively.  As $\log(\lambda_t)=P(t\tilde{\phi})$ (resp.   $\log(\lambda_{t,n})=P_{\Sigma_n}(t\tilde{\phi})$), the problem of finding the Hausdorff dimensions of $J$ (resp. $J_n$) reduces to  finding the values of $t$ (resp. $t_n$) such that $\lambda_{t}=1$ (resp. $\lambda_{t_n,n}=1$).

The proof of Proposition \ref{dimensiondiscrete} relies on a few elementary facts: the maps $t:\mapsto\lambda_{t,n}$ are analytic and non-increasing in $t$, while for a fixed $t$ the sequence $\{\lambda_{t,n}\}_n$ is increasing (and converges to $\lambda_t$), we use Taylor's theorem applied to $\lambda_{t,n}$ about $t=s$ to obtain an approximation of $\lambda_{t,n}$ close to $\lambda_{s,n}$, we then use Theorem \ref{escrates} and let $n\to\infty$ to prove the result.  The main problem then reduces to analysing the behaviour of the first $\lambda_{t,n}^\prime=d/dt(\lambda_{t,n})$ and second $\lambda_{t,n}^{\prime\prime}=d^2/dt^2(\lambda_{t,n})$ derivatives of $\lambda_{t,n}$ which is the focus of the following two technical lemmas.

\begin{lemma}For any $t\geq 0$ we have that $\lim_{n\to\infty} \lambda_{t,n}^\prime=\lambda_{t}^\prime.$\label{derpres}
	\end{lemma}

\begin{proof}We first obtain an explicit formula for $\lambda^\prime_{t,n}$, to do this we follow an argument of Ruelle \cite{Rue04}[p 96. Ex 5.] to prove that for any $t\geq 0$ and $n=1,2,\cdots$
\begin{equation}\label{pertder} \lambda^\prime_{t,n}=-\lambda_{t,n}\int \log|f^\prime| d\mu_{t,n}.\end{equation}

Analogously for the unperturbed operator

\begin{equation}\label{unpertder} \lambda_{t}^\prime=-\lambda_{t}\int \log|f^\prime| d\mu_t.\end{equation}

To see this we take the eigenfunction equation

\begin{equation}\label{eigenfunctioneqn} L_{t,n}g_{t,n}=\lambda_{t,n} g_{t,n}.\end{equation}
	
Differentiating once yields

\begin{equation}\nonumber L_{t,n}^\prime g_{t.n}+L_{t,n} g_{t,n}^\prime=\lambda_{t,n}^\prime g_{t,n}+\lambda_{t,n} g_{t,n}^\prime,\end{equation}

and then integrating with respect to $\nu_{t,n}$ and cancelling terms yields

\begin{equation} \nonumber \lambda^\prime_{t,n} = \int \mathcal{L}_{t,n}^\prime (g_{t,n})d\nu_{t,n} = \int \mathcal{L}_{t,n}(\phi g_{t,n})d\nu_{t,n} 
= \lambda_{t,n}\int \phi d\mu_{t,n}\end{equation}
where $\phi=-\log|f^\prime|$.  This shows (\ref{pertder}), the proof of (\ref{unpertder}) is analogous and the proof is omitted.  

Without loss of generality we may assume that $g_t=1$, that is $\mathcal{L}_{t}1(x)=\lambda_t$.  We decompose the transfer operators $\mathcal{L}_{t,n}$ and $\mathcal{L}_t$ as 

\begin{equation} \nonumber \mathcal{L}_{t,n} = \lambda_{t,n}E_{t,n}+\Psi_{t,n} \,\,\,\, \mathcal{L}_t = \lambda_t E_t + \Psi_t \end{equation}
where $E_{t,n}$ and $E_t$ are projection operator given by 
\begin{equation} E_{t,n}w=\int w d\nu_{t,n} g_{t,n},\,\,\,\,\,\,\,\, E_{t}w=\int w d\nu_t \end{equation} and $\Psi_{t,n}$ (resp. $\Psi_t$) has a spectral radius strictly less than $\lambda_{t,n}$ (resp. $\lambda_t$).

From \cite{KelLiv99} we have that $\lim_{n\to\infty}\||E_{t,n}-E_{t}|\|=0$ and so

\begin{eqnarray}\label{L1converge} \|g_{t,n}-g_t\|_1 &\leq& \|g_{t,n}-g_t\|_w \\
	\nonumber &=& \|(E_{t,n}-E_t)(1)\|_w \leq \||E_{t,n}-E_{t}|\| \|1\|_s \to 0.\end{eqnarray}
	
Finally, to show that $\lambda_{t,n}^\prime\to\lambda_{t}^\prime$ it suffices to show that $E_{t,n}(g_{t,n}\phi)\to E_{t}(\phi)$.  We first show that there exists a constant $c>0$ such that $\|g_{t,n}\phi\|_{\theta,1}\leq c$ for all $n$.  We note that by \cite{KelLiv99}[Corollary 1] that there exists a constant $c>0$ and positive integer $N$ such that $\|E_{t,n}w\|_{\theta,1} \leq c\|E_t w\|_h$ for any $w\in\mathcal{B}_\theta$ and $n\geq N$. In which case 
\begin{eqnarray}\nonumber \|g_{t,n}\phi\|_{\theta,1} &=& |g_{t,n}\phi\|_{\theta,1}+\|g_{t,n}\phi\|_1 \\
	\nonumber &\leq & \|g_{t,n}\|_\infty |\phi|_{\theta,1}+|g_{t,n}|_{\theta,1}\|\phi\|_\infty+\|\phi\|_\infty\|g_{t,n}\|_1 \\
	\nonumber &\leq & 2\|\phi\|_{\theta,1}\|g_{t,n}\|_{\theta,1}+\|\phi\|_\infty\|g_{t,n}\|_w\\
	\nonumber &=& 2\|\phi\|_{\theta,1}\|E_{t,n}1\|_{\theta,1}+\|\phi\|_\infty\|g_{t,n}\|_h\\
	\nonumber &\leq & 2c\|\phi\|_{\theta,1} \|E_{t,n}1\|_h+\|\phi\|_\infty\|g_{t,n}\|_h \\
	\nonumber &=& (2c\|\phi\|_s+\|\phi\|_\infty)\|g_{t,n}\|_h.\end{eqnarray}
We observe that $\|g_{t,n}-1\|_1\to 0$, implies that $\|g_{t,n}-1\|_h\to 0$ and so $\|g_{t,n}\|_{\theta,1}$ is bounded.  Next, we note that

\begin{eqnarray}\nonumber \|E_{t,n}(g_{t,n}\phi)-E_t(\phi)\|_1 & = & |\|E_{t,n}-E_t\||\|g_{t,n}\phi\|_{\theta,1}+\|E_t(\phi(g_{t,n}-1))\|_1 \\
\nonumber & \leq & c|\|E_{t,n}-E_t\|| + \|\phi\|_\infty \|E_t\|_1\|g_{t,n}-1\|_1.\end{eqnarray}
Both terms tend to zero by equation (\ref{L1converge}) and \cite{KelLiv99}.  This completes the proof. 
	 \end{proof}

\begin{lemma}For any $s>0$ there exists $\delta>0$ such that $\sup_{n\geq 1}\sup_{t\in (s-\delta,s+\delta)} \lambda^{\prime\prime}_{t,n} < \infty$.\label{secderpres}\end{lemma}

\begin{proof}We first obtain an expression for $\lambda_{t,n}^{\prime\prime}$.  Fix a positive integer $N$, taking the eigenfunction equation ${\mathcal L}_{t,n}^Ng_{t,n}=\lambda^{N}_{t,n}g_{t,n}$ and differentiating twice, integrating with respect to $\nu_{t,n}$ and cancelling yields
	
\begin{eqnarray}\nonumber \lambda_{t,n}^{-1}\lambda_{t,n}^{\prime\prime}&=&\frac{1}{N}\left[\int(\phi^N)^2g_{t,n}d\nu_{t,n}-N(N-1)(\lambda_{t,n}^{-1}\lambda_{t,n}^\prime)^2\right]\\
	\nonumber  & & \,\,\,\,\,\,\,\,\,\,\,\,\,\,\,\,\,\,\,\,+2\left[\frac{1}{N}\int \phi^N g_{t,n}^\prime d\nu_{t,n}-\lambda_{t,n}^{-1}\lambda_{t,n}^\prime\int g_{t,n}^\prime d\nu_{t,n}\right].\end{eqnarray}	
	
We observe that as $d\mu_{t,n}=g_{t,n}d\nu_{t,n}$ is strong mixing that this second term tends to zero as $N\to\infty$, and thus 

\begin{eqnarray}\label{secondder} \lambda_{t,n}^{-1}\lambda_{t,n}^{\prime\prime}=\lim_{N\to\infty}\frac{1}{N}\left[\int(\phi^N)^2g_{t,n}d\nu_{t,n}-N(N-1)(\lambda_{t,n}^{-1}\lambda_{t,n}^\prime)^2\right].\end{eqnarray}

We now estimate the term $N^{-1}\int (\phi^N)^2 g_{t,n} d\nu_{t,n}$: expanding the term $(\phi^N)^2$ and using the dual identity $\mathcal{L}_{t,n}^{*}(\nu_{t,n})=\lambda_{t,n}\nu_{t,n}$ yields for $n$ large enough

\begin{eqnarray}\nonumber N^{-1}\int (\phi^N)^2 g_{t,n} d\nu_{t,n} &=& \sum_{i=0}^{N-1}\sum_{j=0}^{N-1} \int g_{t,n}\phi\circ\sigma^{i}\phi\circ\sigma^{j} d\nu_{t,n} \\
	\nonumber  &=& \|\phi\|_2^2+\frac{2}{N}\sum_{k=0}^{N-1}(N-k)\int g_{t,n}\phi\phi\circ\sigma^k d\nu_{t,n}\\
	\label{secondder1} &=& \|\phi\|_2^2+\frac{2}{N}\sum_{k=0}^{N-1}(N-k)\lambda_{t,n}^{-k}\int \mathcal{L}^k_{t,n}(g_{t,n}\phi)\phi d\nu_{t,n}.\end{eqnarray}
We apply the decomposition $L_{t,n}=\lambda_{t,n}E_{t,n}+\Psi_{t,n}$ along with Proposition \ref{resolvantbound1} to equation (\ref{secondder1}) to obtain 
\begin{eqnarray}\nonumber N^{-1}\int (\phi^N)^2 g_{t,n} d\nu_{t,n} &=& \|\phi\|_2^2+\frac{2}{N}\sum_{k=0}^{N-1}(N-k) \int E_{t,n}(g_{t,n}\phi)\phi  +\lambda_{t,n}^{-k} \Psi^k_{t,n}(g_{t,n}\phi)\phi d\nu_{t,n} \\
\label{secondder2} &=& \|\phi\|_2^2+ (N+1)\left(\int \phi g_{t,n} d\nu_{t,n}\right)^2 \\
\nonumber & & \,\,\,\,\,\,\,\,\,\,\,\,\,+  \frac{2}{N}\sum_{k=1}^{N-1}(N-k)\lambda_{t,n}^{-k} \int  \Psi^k_{t,n}(g_{t,n}\phi)\phi d\nu_{t,n}.\end{eqnarray} 

We note that $\left(\int \phi g_{t,n} d\nu_{t,n}\right)^2=(\lambda_{t,n}^{-1}\lambda_{t,n}^\prime)^2$ and so combining equations (\ref{secondder}) and (\ref{secondder2}) we obtain

\begin{equation}\nonumber \lambda_{t,n}^{-1}\lambda_{t,n}^{\prime\prime}=\|\phi\|_2^2+2(\lambda_{t,n}^{-1}\lambda_{t,n}^\prime)^2+\lim_{N\to\infty}\frac{2}{N} \sum_{k=0}^{N-1}(N-k)\lambda_{t,n}^{-k}\int  \Psi^k_{t,n}(g_{t,n}\phi)\phi  d\nu_{t,n}.\end{equation}
	
Finally we observe that the perturbation $t\mapsto\mathcal{L}_{t}$ is analytic and so for any $q>0$ such that $spec(\mathcal{L}_s)\setminus\{\lambda_s\}\subset B(0,q)$ there exists a positive integer $M$ and $\delta>0$ such that $\lambda_{t.n}>q$ and $spec(\mathcal{L}_{t,n})\setminus\{\lambda_{t,n}\}\subset B(0,q)$ for all $n\geq M$ and $t\in (s-\delta,s+\delta)$.  Combining this observation with Proposition \ref{resolvantbound1} completes the proof.\end{proof}

We now prove Proposition \ref{dimensiondiscrete}:

\begin{proof}We begin by proving that $s-s_n=O(\mu(I_n))$, to see this we observe the map $t\mapsto\lambda_{t,n}$ is analytic, and so using Taylor's theorem we may write 
	\begin{equation}\label{claim1} \lambda_{s_n,n}=1=\lambda_{s,n}+\lambda^\prime_{\xi_n,n}(s_n-s).\end{equation}
for some $\xi_n\in(s_n,s)$.  We note that Theorem \ref{mainpert} and Lemma \ref{derpres} prove the claim.  Next, we use Taylor's theorem once again to see that 
\begin{equation}\nonumber \lambda_{s_n,n}=1=\lambda_{s,n}+\lambda_{s,n}^\prime(s_n-s)+\lambda_{\xi_n,n}^{\prime\prime}O(\mu(I_n)^2)\end{equation}
for $\xi_n\in(s_{n},s)$.  Rearranging yields
\begin{equation}\nonumber \frac{s-s_n}{\mu(I_n)}=\frac{1}{-\lambda_{s,n}^\prime}\left[\frac{1-\lambda_{s,n}}{\mu(I_n)}+\lambda_{\xi_n,n}^{\prime\prime}O(\mu(I_n))\right].\end{equation}
Finally we let $n\to\infty$ observing that the right hand converges by Lemmas \ref{derpres} and \ref{secderpres}.  This completes the proof.
\end{proof} 

We note that as in the case of escape rates Proposition \ref{dimensiondiscrete} generalises easily to the case of finite unions of symbolic holes.  We now prove Theorem \ref{dimension}.

\begin{proof}Let $\{\epsilon_n\}_n$ be any monotonic sequence with $\epsilon_n\to 0$.  Fix $\eta>0$ and choose $U_n \subset B(z,\epsilon_n)\subset V_n$ which consist of finite unions of refinements of the Markov partition $\mathcal{R}$ such that

\begin{equation}\label{compare} (1-\eta)\mu(V_n)\leq \mu (B(z,\epsilon_n)) \leq (1+\eta)\mu( U_n).\end{equation}

From the proof of Theorem \ref{escratesexp}, it is clear that we may choose the families $\{U_n\}_n$ and $\{V_n\}_n$ so that they satisfy the hypotheses of Proposition \ref{dimensiondiscrete}.  Let $\underline{s}_n$ (resp. $\overline{s}_n$ denote the Hausdorff dimension of the non-trapped set with respect to the hole $U_n$ (resp. $V_n$.  Monotinicity of the Hausdorff dimension along with equation (\ref{compare}) yields
\begin{equation}\nonumber \frac{1}{1+\eta}\frac{s-\underline{s}_n}{\mu(U_n)} \leq \frac{s-s_n}{\mu(B(z,\epsilon_n))} \leq \frac{1}{1-\eta}\frac{s-\overline{s}_n}{\mu(V_n)},\end{equation}
for any $n$, combining this with Proposition \ref{dimensiondiscrete} and letting $\eta\to 0$ completes the proof.\end{proof}

\ack{We would like to thank the referee for his careful reading of the original version of this paper and his numerous careful comments.}

\bibliographystyle{abbrv}
\bibliography{escaperates}
\end{document}